\documentclass[12pt]{amsart}
\usepackage{amssymb}
\usepackage{amsmath}
\numberwithin{equation}{section}
\usepackage{color}
\usepackage{graphicx}
\usepackage{epstopdf}

\newtheorem{thm}{Theorem}[section]
\newtheorem{lemma}[thm]{Lemma}
\newtheorem{cor}[thm]{Corollary}
\newtheorem{prop}[thm]{Proposition}
\newtheorem{con}[thm]{Conjecture}

\theoremstyle{definition}

\newtheorem{defn}[thm]{Definition}

\theoremstyle{remark}
\newtheorem{remark}[thm]{Remark}

\theoremstyle{rem}

\renewcommand{\S}{\mathfrak S}
\newcommand{\s}{\sigma}

\newcommand\C{{\mathbb{C}}}

\newcommand\Q{{\mathbb Q}}

\newcommand\F{{\mathbb{F}}}
\newcommand\PP{{\mathbb{P}}}
\newcommand\N{{\mathbb{N}}}

\newcommand\bq{\begin{equation}}
\newcommand\eq{\end{equation}}
\newcommand\beq{\begin{eqnarray*}}
\newcommand\eeq{\end{eqnarray*}}
\newcommand\ben{\begin{enumerate}}
\newcommand\een{\end{enumerate}}
\newcommand\bit{\begin{itemize}}
\newcommand\eit{\end{itemize}}

\newcommand\des{{\rm des}}
\newcommand\exc{{\rm exc}}
\newcommand\inv{{\rm inv}}
\newcommand\maj{{\rm maj}}
\newcommand\comaj{{\rm comaj}}

\newcommand\sg{{\mathfrak S}}
\newcommand\Des{{\rm DES}}
\newcommand\Exd{{\rm DEX}}
\newcommand\Exc{{\rm EXC}}
\newcommand\fix{{\rm fix}}
\newcommand\ch{{\rm ch}}

\newcommand\x{{\mathbf x}}
\newcommand\bbar{{\rm bar}}
\newcommand\Exp{{\rm Exp}}
\newcommand\wt{{\rm wt}}

\newcommand\Com{{\rm Com}}
\newcommand\Par{{\rm Par}}

\def\Lambda{{\bf ps}}

\def\ov{\overline}

\def\rh{\tilde{H}}

\def\xx{{\mathbf x}}

\begin{document}

\title[Eulerian quasisymmetric functions]
{Eulerian quasisymmetric functions}
\author[Shareshian]{John Shareshian$^1$}
\address{Department of Mathematics, Washington University, St. Louis, MO 63130}
\thanks{$^{1}$Supported in part by NSF Grants
 DMS 0300483 and DMS 0604233, and the Mittag-Leffler Institute}
\email{shareshi@math.wustl.edu}

\author[Wachs]{Michelle L. Wachs$^2$}
\address{Department of Mathematics, University of Miami, Coral Gables, FL 33124}
\email{wachs@math.miami.edu}
\thanks{$^{2}$Supported in part by NSF Grants
DMS 0302310 and DMS 0604562, and the Mittag-Leffler Institute}

\subjclass[2000]{05A15, 05A30, 05E05}

\date{June 1, 2008; version August 15, 2010}

\dedicatory{}

\begin{abstract} We introduce  a   family of  quasisymmetric functions called {\em Eulerian quasisymmetric functions},
which specialize to   enumerators for the joint distribution of the permutation statistics, major index and excedance number on permutations of fixed cycle type.
 This family is analogous to a family of quasisymmetric functions that Gessel and Reutenauer used to study the joint distribution of major index and descent number on permutations of fixed cycle type.
 Our central result is a formula for
the generating function for the Eulerian quasisymmetric  functions, which specializes to
a new  and surprising
$q$-analog of a classical formula of Euler for the exponential generating
function of the Eulerian polynomials.  This $q$-analog computes the joint distribution of  excedance number and major index, the only of the four important Euler-Mahonian distributions that had not yet been computed.   Our  study of the Eulerian quasisymmetric functions   also yields  results  that include the descent statistic and  refine   results of Gessel and Reutenauer.   We also obtain $q$-analogs,  $(q,p)$-analogs and quasisymmetric function analogs of classical results on the symmetry and unimodality of the Eulerian polynomials.  Our Eulerian quasisymmetric functions refine  symmetric functions that have occurred in various representation theoretic and enumerative contexts including MacMahon's study of multiset derangements,   work of  Procesi and Stanley   on  toric varieties of  Coxeter complexes,   Stanley's work on  chromatic symmetric functions, and the work of the authors on the homology of a certain poset introduced by Bj\"orner and Welker. 
 \end{abstract}

\maketitle

\vbox{
\tableofcontents
}
\section{Introduction}
Through our study \cite{sw2} of the homology of a certain partially ordered set introduced by
Bj\"orner and Welker, we  have discovered a
remarkable  $q$-analog of a classical formula for the Eulerian polynomials.   
  The
 Eulerian polynomials can be interpreted combinatorially as enumerators of permutations by  the permutation statistic,  excedance number.
 Our $q$-analog is a formula for the joint
 distribution of the excedance  statistic and  the major index, the only of the four important Euler-Mahonian distributions  yet to be addressed in the literature.
 Although poset topology led us to
 conjecture our formula, it is the theory of symmetric functions and quasisymmetric functions  that provides  the proof  of our original formula, as well as  of more refined versions involving additional permutation statistics.  In the research announcement \cite{sw}  these results and sketches of their  proofs were presented.  Here we provide the details and  additional results.

 The modern study of permutation statistics began with the work of
Major Percy McMahon \cite[Vol. I, pp.~135, 186; Vol. II, p.~viii] {mac1}, \cite{mac2}.
It involves the enumeration of  permutations
according to   natural statistics.    A permutation statistic is
simply a function from the union of all
the symmetric groups $\S_n$ to the set of
nonnegative integers.     MacMahon studied four fundamental
permutation statistics,  the inversion index ($\inv$), the major index ($\maj$), the
descent number ($\des$), and the excedance number ($\exc$), which are defined in Section~\ref{statsec}.

MacMahon observed in \cite[Vol. I, p. 186]{mac1} the now well
known result that the statistics $\des$ and $\exc$ are
equidistributed, that is,
\begin{equation*} \label{euler}
A_n(t):=\sum_{\sigma \in \sg_n}t^{\des(\sigma)}=\sum_{\sigma \in \sg_n}t^{\exc(\sigma)},
\end{equation*}
for every positive integer $n$.
The coefficients of the polynomials $A_n(t)$ were studied by
Euler, and are called {\it Eulerian numbers}.  The polynomials are
known as {\it Eulerian polynomials}.  (Note that it is common in
the literature to define the Eulerian polynomials to be
$tA_n(t)$.) Any permutation statistic that is equidistributed with
$\des$ and $\exc$ is called an {\it Eulerian statistic}.  Eulerian
numbers and polynomials have been extensively studied (see for
example \cite{fs} and \cite{kn}).  Euler proved (see \cite[p.
39]{kn}) the generating function formula
\begin{equation}  \label{expgen}
1+\sum_{n \geq 1}A_n(t)\frac{z^n}{n!}=\frac{1-t}{e^{z(t-1)}-t}.
\end{equation}

For a positive integer $n$, the polynomials $[n]_q$ and $[n]_q!$
are defined as
\[
[n]_q:=1+q+\ldots+q^{n-1}
\]
and
\[
[n]_q!:=\prod_{j=1}^{n}[j]_q.
\]
Also set $[0]_q := 0$ and $[0]_q!:=1$.
MacMahon proved in \cite{mac2} the first equality in the equation
\begin{equation*} \label{majeq}
\sum_{\sigma \in \S_n}q^{\maj(\sigma)}=[n]_q!=\sum_{\sigma \in
\S_n}q^{\inv(\sigma)}
\end{equation*}
after the second equality had been obtained in \cite{rod} by
Rodrigues.  A permutation statistic that is equidistributed with
$\maj$ and $\inv$ is called a {\it Mahonian} statistic.

Much effort has been put into the examination of joint
distributions of pairs of permutation statistics, one Eulerian and
one Mahonian (for a sample see,
\cite{bs,br,c,csz,foa3,fs2,fz,gar,gg,grem,gr,hag,rrw,ra,sk,st,st1,wa2}).
One beautiful result on such joint distributions is found in the
paper \cite{st1} of Stanley. For permutation statistics ${\mathsf
{f_1}},\ldots,{\mathsf {f_k}}$ and a positive integer $n$, define
the polynomial
\[
A_n^{{\mathsf {f_1,\ldots,f_k}}}(t_1,\ldots,t_k):=\sum_{\sigma \in
\S_n}t_1^{{\mathsf {f_1}}  (\sigma) } t_2^{{\mathsf {f_2}}
(\sigma)}\cdots t_k^{{\mathsf {f_k}}  (\sigma)}.
\]
Also, set
\[
A_0^{{\mathsf {f_1,\ldots,f_k}}}(t_1,\ldots,t_k):=1.
\]
Stanley showed that if we define
\[
\Exp_q(z):=\sum_{n \geq 0}q^{{n} \choose {2}}\frac{z^n}{[n]_q!}
\]
then we have the $q$-analogue
\begin{equation*} \label{staneq}
\sum_{n \geq 0}A_n^{\inv,\des}(q,t)\frac{z^n}{[n]_q!}=\frac{1-t}{\Exp_q(z(t-1))-t}
\end{equation*}
of (\ref{expgen}).
\\
\\
Although there has been much work on the joint distributions of $(\maj,\des)$, $(\inv,\exc)$ and $(\inv,\des)$ there were to our knowledge no results about
$A^{\maj,\exc}_n(q,t)$ in the existing literature prior to \cite{sw}, where our work was first  announced.   Our main result on
these polynomials is as follows.  Set
\[
\exp_q(z):=\sum_{n \geq 0}\frac{z^n}{[n]_q!}.
\]

\begin{thm} \label{expgenth}
We have
\begin{equation} \label{expgeneq}
\sum_{n \geq 0}A^{\maj,\exc}_n(q,t)\frac{z^n}{[n]_q!}=\frac{(1-tq)\exp_q(z)}{\exp_q(ztq)-tq\exp_q(z)}.
\end{equation}
\end{thm}

It is well-known that $\exp_q(z)$ is a specialization (essentially the stable principal specialization) of the symmetric function
$$ H(z)=H(\xx,z):=\sum_{n \geq 0}h_n(\xx)z^n,
$$
where $h_n$ is the complete homogeneous symmetric function of degree $n$ (details are given in Section~\ref{statsec}).
Hence the right
hand side of (\ref{expgeneq}) is a specialization of the symmetric function
$$\frac{(1-tq)H(z)}{H(ztq)-tqH(z)}.$$

We introduce a family of  quasisymmetric functions $Q_{n,j,k}(\mathbf x)$, called  Eulerian quasisymmetric functions, whose
 generating function $$\sum_{n,j ,k\ge 0} Q_{n,j,k}(\x) t^j r^k z^n$$
specializes to
$$\sum_{n\ge 0} \sum_{\s \in \S_n} q^{\maj(\s) - \exc(\s)} t^{\exc(\s)}r^{\fix(\s)} {z^n \over [n]_q!}, $$
   where $\fix(\sigma) $ is the number of fixed points of  $\sigma$.
The $Q_{n,j,k}$ are defined as sums of fundamental quasisymmetric functions that we associate (in a nonstandard way) with permutations  $\s \in \mathfrak S_n$ satisfying  $\exc(\s) = j$ and $\fix(\s) = k$.  (These Eulerian quasisymmetric functions should not be confused with the quasisymmetric functions associated with Eulerian posets discussed in      \cite{bhw,eh}.)  Our central result is the following theorem.

\begin{thm}\label{introsymgenth}
We have
\begin{eqnarray}
\label{introsymgenth1}\sum_{n,j,k \geq 0}Q_{n,j,k}(\xx)t^jr^kz^n&=&\frac{(1-t)H(rz)}{H(zt)-tH(z)}\\ \label{introsymgenth2}&=& \frac{H(rz)}{1-\sum_{n\ge 2}t[n-1]_t h_nz^n}.
\end{eqnarray}
\end{thm}

The proof of Theorem \ref{introsymgenth} appears in Section
\ref{identsec}.  It depends on combinatorial bijections and
involves nontrivial extension of techniques introduced by  Gessel
and Reutenauer in \cite{gr}, D\'esarm\'enien and Wachs in
\cite{dw} and Stembridge in \cite{stem1}.   Gessel and Reutenauer
construct a bijection from  multisets of  primitive circular words
of fixed content to permutations in order   to enumerate
permutations with a given  descent set and  cycle type. This
bijection, which is related to Stanley's theory of P-partitions
\cite{st},  has also proved useful in papers of D\'esarm\'enien
and Wachs \cite{dw2,dw},  Diaconis, McGrath and Pitman \cite{dmp},
Hanlon \cite{h}, and Wachs \cite{wa3}.   Reiner \cite{re}
introduced a type B analog. Here we introduce a bicolored version
of the Gessel-Reutenauer bijection.

By specializing  Theorem~\ref{introsymgenth} we get the following strengthening of Theorem~\ref{expgenth}.
\begin{cor} \label{expgenthfix}
We have
\begin{equation} \label{expgeneqfix}
\sum_{n \geq 0}A^{\maj,\exc,\fix}_n(q,t,r)\frac{z^n}{[n]_q!}=\frac{(1-tq)\exp_q(rz)}{\exp_q(ztq)-tq\exp_q(z)} .
\end{equation}
\end{cor}

By setting $t=1$ in (\ref{expgeneqfix}) one  obtains a  formula of Gessel and Reutenauer \cite{gr}.
 By setting $r=0$, we obtain a new  $q$-analog of  a formula of Brenti \cite[Proposition 5]{bre1} for the $\exc$-enumerator of derangements.
  Using Corollary~\ref{expgenthfix}, we exhibit in Section~\ref{altsec} explicit formulae, both recursive and closed, for $A_n^{\maj,\exc,\fix}$.

We will show  that a different specialization (essentially the nonstable principal specialization) of Theorem~ \ref{introsymgenth} readily yields a further extension of Theorem~\ref{expgenth} that Foata and Han obtained after seeing a preprint containing our work \cite{sw}.

  \begin{cor}[Foata and Han \cite{FH}] \label{foha}
We have
\[
\sum_{n \geq
0}A_n^{\maj,\des,\exc,\fix}(q,p,t,r)\frac{z^n}{(p;q)_{n+1}}=\sum_{m
\geq
0}p^m\frac{(1-qt)(z;q)_m(ztq;q)_m}{((z;q)_m-tq(ztq;q)_m)(zr;q)_{m+1}},
\]
where\[
(a;q)_n:= \left\{\begin{array}{ll} 1 & \mbox{if } n=0 \\
(1-a)(1-aq)\ldots(1-aq^{n-1}) & \mbox{if } n \geq 1.
\end{array}
\right.
\]
\end{cor}

It is well-known that the Eulerian polynomials are symmetric and unimodal. In Section~\ref{secQ} we use Theorem~\ref{introsymgenth} to obtain
$q$-analogs of these results for $A_n^{\maj,\exc} (q,q^{-1}t)$, and  $(q,p)$-analogs for the derangement enumerator $A_n^{\maj,\des,\exc,\fix}(q,p,q^{-1}t,0)$.   Symmetric function analogs of these results for the Eulerian quasisymmetric functions $Q_{n,j,k}$ are also obtained.

Essential to our proof of Theorem~\ref{introsymgenth} is a refinement of   $Q_{n,j,k}$ which is  quite interesting in its own right.  Given a partition $\lambda $ of $n$, we define  the cycle type Eulerian quasisymmetric function $Q_{\lambda,j}$ to be the sum of the fundamental quasisymmetric functions  associated with permutations that have cycle type $\lambda$ and  $j$ excedances.  In Section~\ref{secQ} we prove
that  $Q_{\lambda,j}$ is, in fact, a symmetric function.  We also obtain another symmetry result, specifically: $Q_{\lambda,j} = Q_{\lambda,n-k-j}$ for all $\lambda \vdash n$, where $k$ is the number of parts of $\lambda$ equal to 1 and $j=0,\dots,n-k$.  By specializing this result we conclude that the cycle type Eulerian polynomials $A_{\lambda,j}^{\maj,\des,\exc}(q,p,q^{-1}t):=\sum_{\sigma} q^{\maj(\s) -\exc(\s)}p^{\des(\s)} t^{\exc(\s)}$, where $\s$ ranges over all permutations that have cycle type $\lambda$ and $j$ excedances, is a symmetric polynomial in $t$.  We conjecture\footnote{All the conjectures referred to in this paragraph were recently proved jointly with A. Henderson. See Section~\ref{newdevsec} (New developments).} that these polynomials are unimodal as well.
 It follows from Theorem~\ref{introsymgenth} that the Eulerian quasisymmetric functions $Q_{n,j,k}$ are h-positive. We show that this does not hold for the more refined cycle type Eulerian quasisymmetric functions $Q_{\lambda,j}$, but  conjecture that these are Schur positive.  We also conjecture a symmetric function analog of unimodality for the polynomials $\sum_j Q_{\lambda,j}t^j$. 
 
  Of particular interest are the
 $Q_{\lambda,j}$ when the partition $\lambda$ consists of a single part.  In Section~\ref{repthsec} it is observed that the $Q_{\lambda,j}$ for general $\lambda$ can be expressed via plethysm in terms of  these.   We present a conjecture\footnote{This conjecture has now been proved by Sagan, Shareshian and Wachs [48].} describing  the coefficients of the expansion of  $Q_{\lambda,j}$  in the power sum symmetric function basis  for the case that $\lambda$ consists of a single part.  We also derive results  on the virtual representation of the symmetric group whose Frobenius characteristic is $Q_{\lambda,j}$ for this case.
 
 The  symmetric functions $Q_{\lambda,j}$  resemble  the symmetric functions $L_\lambda$ studied by
Gessel and Reutenauer in \cite{gr} in their work on quasisymmetric
functions and permutation statistics.  However our $Q_{\lambda,j}$ are not
refinements of the  $L_\lambda$. Indeed, $L_\lambda$ is the
Frobenius characteristic of a representation induced from a linear
character of the centralizer of a permutation of cycle type
$\lambda$.  On the other hand, we show that if $\lambda= (n)$, where $n \ge 3$,  then $\sum_jQ_{\lambda,j}$ is
the Frobenius characteristic of a  virtual representation (conjecturally, an actual representation) whose character
takes nonzero values on elements that do not commute with any
element of cycle type $\lambda$ (see Corollary~\ref{cvalcor}).

The symmetric function on the right hand side of (\ref{introsymgenth1})  and  (\ref{introsymgenth2}) refines  symmetric functions that have been studied earlier in the literature.  Occurrences of these symmetric functions include  enumerators for multiset derangements studied by MacMahon \cite[Sec. III, Ch. III]{mac1} and Askey and Ismail \cite{ai};  enumerators  for words with no adjacent repeats studied by Carlitz, Scoville and  Vaughan \cite{csv},
 Dollhopf, Goulden and Greene \cite{dgg} and Stanley \cite{st4};  chromatic symmetric functions  of Stanley \cite{st4}; and the  Frobenius characteristic of the representation of the symmetric group on the degree $2j$ cohomology of the toric variety $X_n$ associated to the Coxeter complex of the symmetric group $\mathfrak S_n$ studied by Procesi \cite{pr}, Stanley \cite{st2}, Stembridge \cite{stem1,stem2}, and Dolgachev and Lunts \cite{dl}.   It is a consequence of our work that these symmetric functions   have nice interpretations as sums of fundamental quasisymmetric functions.   Representations with Frobenius characteristic $Q_{n,j}$ also occur in the recent paper of the authors on poset topology \cite{sw2}.
Indeed, it was our study  of the homology of a certain poset introduced by Bj\"orner and Welker \cite{bw} that led us to conjecture Theorems~\ref{expgenth} 
and~\ref{introsymgenth} in the first place.   These connections and others are discussed in Section~\ref{othersec}.

\section{Permutation statistics and quasisymmetric functions} \label{statsec}

For $n \ge 1$, let $\S_n$ be the symmetric group on the set
$[n]:=\{1,\ldots,n\}$.  A permutation $\sigma \in \S_n$ will be
represented here in two ways, either as a function that maps $i
\in [n]$ to $\sigma(i)$, or in one line notation as
$\sigma=\sigma_1\ldots\sigma_n$, where $\sigma_i=\sigma(i)$.

 The
{\it descent set} of  $\s\in \sg_n$ is
\[
\Des(\s):=\{i \in [n-1]:\sigma_i>\sigma_{i+1}\},
\]
and the {\it excedance set} of $\sigma$ is
\[
\Exc(\sigma):=\{i \in [n-1]:\sigma_i>i\}.
\]

Now we define  the two basic Eulerian permutation statistics.
The {\it descent number} and {\it excedance number} of $\sigma\in \sg_n$
are, respectively,
\[
\des(\sigma):=|\Des(\sigma)|
\]
and
\[
\exc(\sigma):=|\Exc(\sigma)|.
\]
 For
example, if $\sigma = 32541$, written in one line notation, then
$$\Des(\s) = \{1,3,4\} \qquad \mbox{ and }\qquad \Exc(\s) = \{1,3\};$$
 hence $\des(\s) = 3 $ and $\exc(\s) = 2$.  If $i \in \Des(\s)$ we say
 that $\s$ has a descent at $i$.  If $i \in \Exc(\s)$ we say that
 $\s_i$ is an excedance of $\s$ and that $i$ is an excedance position.

Next we define the two basic Mahonian permutation statistics.  The
{\it inversion index} of $\sigma \in \S_n$ is
\[
\inv(\sigma):=|\{(i,j):1 \leq i<j \leq n \mbox { and }\sigma_i>\sigma_j\}|,
\]
and the {\it major index} of $\sigma$ is
\[
\maj(\sigma):=\sum_{i \in \Des(\sigma)}i.
\]
 For
example, if $\sigma = 32541$ then $\inv(\s) = 6$ and $\maj(\s) = 8$.

We review  some basic facts of Gessel's theory of quasisymmetric functions; a good reference is  \cite[Chapter 7]{st3}. A {\it quasisymmetric function} is a formal power series $f(\xx)=f(x_1,x_2,\ldots)$ of finite degree with rational coefficients in the infinitely many variables $x_1,x_2,\ldots$ such that  for any $a_1,\dots,a_k \in \PP$,  the coefficient of $x_{i_1}^{a_1}\dots x_{i_k}^{a_k}$ equals the  coefficient of $x_{j_1}^{a_1}\dots x_{j_k}^{a_k}$ whenever $i_1 < \dots < i_k$ and $j_1 < \dots < j_k$.  Thus each symmetric function is quasisymmetric.

For a positive integer $n$ and $S \subseteq [n-1]$, define
\[
F_{S,n}=F_{S,n}(\xx):=\sum_{\scriptsize\begin{array}{c}i_1 \geq \ldots \geq i_n \geq 1\\ j \in S \Rightarrow i_j>i_{j+1}\end{array}}x_{i_1}\dots x_{i_n}
\] and let
$F_{\emptyset,0}=1$.
Each  $F_{S,n}$ is a quasisymmetric function.  The  set $\{F_{S,n} : S \subseteq [n-1], n\in \N\}$ is a basis
for the algebra $\mathcal Q$ of quasisymmetric functions  and $F_{S,n}$ is called {\em a fundamental} quasisymmetric function.
If $S = \emptyset $ then $F_{S,n}$ is the complete homogeneous symmetric function $h_n$ and if $S=[n-1]$ then $F_{S,n}$ is the elementary symmetric function $e_n$.

We review two important ways to specialize a quasisymmetric function. Let
$\Q[q]$ denote the ring of polynomials in variable $q$ with coefficients in  $\Q$ and let
$\Q[[q]]$ denote the ring of formal power series in variable $q$ with coefficients in $\Q$.  The  {\em stable principal  specialization} is  the ring homomorphism
$\Lambda:\mathcal Q \to \Q[[q]]$ defined by $$\Lambda(x_i) =  q^{i-1} ,$$ and the {\em  principal  specialization} of order $m$ is the ring
homomorphism $\Lambda_m:\mathcal Q \to \Q[q]$ defined by
$$\Lambda_m(x_i) = \begin{cases} q^{i-1} &\mbox{if } 1 \le i \le m \\ 0 &\mbox{if } i > m \end{cases} .$$  
We can extend the ring homomorphism $\Lambda$  to a ring homomorphism
$\Lambda: \mathcal Q[[z_1,\dots,z_m]] \to \Q[[q,z_1,\dots,z_m]] $  defined by $$\Lambda \left(\sum_{n_1,\dots,n_m \ge 0} G_{n_1,\dots,n_k} z_1^{n_1}\dots z_m^{n_m} \right) = \sum_{n_1,\dots,n_m \ge 0} \Lambda(G_{n_1,\dots,n_k}) z_1^{n_1}\dots z_m^{n_m}.$$  The ring homomorphism $\Lambda_m$ can be similarly extended.

\begin{lemma}[{\cite[Lemma 5.2]{gr}}] \label{desspec}For all $n \ge 1$ and $S \subseteq [n-1]$, we have
\begin{equation}\label{spec1}\Lambda (F_{S,n})= \frac {q^{\sum_{i\in S} i}} {(q;q)_n} \end{equation} and
 \begin{equation} \label{spec2}\sum_{m \ge 0} \Lambda_m( F_{S,n}) p^m = {p^{|S|+1} q^{\sum_{i\in S} i} \over (p;q)_{n+1}}.\end{equation} \end{lemma}

It follows from Lemma~\ref{desspec} that
\begin{equation} \label{stabh}  \Lambda (h_n) = \frac 1{(1-q)(1-q^2) \dots (1-q^n)}\end{equation} for all $n$ and therefore
\begin{equation}\label{spech}\Lambda (H(z(1-q)) )= \exp_q(z).\end{equation}

The standard way to connect quasisymmetric functions with  permutation statistics is by associating  the fundamental quasisymmetric function $F_{\Des(\sigma),n}$ with  $\sigma \in \sg_n$.   By Lemma~\ref{desspec}, we have for any subset $A \in \sg_n$
$$\Lambda\left(\sum_{\sigma \in A} F_{\Des(\sigma),n}\right) = {1 \over (q;q)_n} \sum_{\s \in A} q^{\maj(\s)},$$
and
$$\sum_{m\ge 0} \Lambda_m\left(\sum_{\sigma \in A} F_{\Des(\sigma),n}\right) p^m = {1 \over (p;q)_{n+1}}\sum_{\s \in A} p^{\des(\s)+1}q^{\maj(\s)}.$$
Gessel and Reutenauer \cite{gr} used this to study the $(\maj,\des)$-enumerator for permutations of a fixed cycle type.

Here we introduce a new way to associate   fundamental quasisymmetric functions with permutations.  For $n \ge 1$, we set \[
[\ov{n}]:=\{\ov{1},\ldots,\ov{n}\}
\]
and totally order the alphabet $[n] \cup [\ov{n}]$ by
\begin{equation} \label{order1}
\ov{1}<\ldots<\ov{n}<1<\ldots<n.
\end{equation}
For a permutation $\sigma=\sigma_1\ldots\sigma_n \in \S_n$, we
define $\ov{\sigma}$ to be the word over  alphabet $[n] \cup
[\ov{n}]$ obtained from $\sigma$ by replacing $\sigma_i$ with
$\ov{\sigma_i}$ whenever $i \in \Exc(\sigma)$.  For example, if
$\sigma={\rm {531462}}$ then $\ov{\sigma}={\rm
{\ov{5}\ov{3}14\ov{6}2}}$.  We define a descent in a word
$w=w_1\ldots w_n$ over any totally ordered alphabet to be any $i
\in [n-1]$ such that $w_i>w_{i+1}$ and let $\Des(w)$ be the set of
all descents of $w$.  Now, for $\sigma \in \S_n$, we define
\[
\Exd(\sigma):=\Des(\ov{\sigma}).
\]
For example, $\Exd({\rm {531462}})=\Des({\rm {\ov{5}\ov{3}14\ov{6}2}})=\{1,4\}$.

\begin{lemma} \label{exdlem}  For all $\s \in \S_n$,
\begin{equation}\label{exd}  \sum_{i \in \Exd(\s)} i = \maj(\s) - \exc(\s),\end{equation}
and
\begin{equation}\label{exdsize} |\Exd(\s)| =  \begin{cases} \des(\s)  &\mbox{if } \s(1) = 1 \\  \des(\s)-1 &\mbox {if } \s(1) \ne 1 \end{cases} \end{equation}
\end{lemma}

\begin{proof} Let $$J(\s) = \{i \in [n-1]:  i \notin \Exc(\s) \,\,\, \& \,\,\, i+1 \in \Exc(\s)\},$$
and let
$$ K(\s) = \{i \in [n-1]:  i\in \Exc(\s) \,\,\, \& \,\,\, i+1 \notin \Exc(\s)\}.$$
If $i \in J(\s)$ then $\s(i) \le i < i+1 < \s(i+1)$.  Hence $i \notin \Des(\s)$, but $i \in \Exd(\s)$.  If $i \in K(\s)$ then $\s(i)\ge  i+1 \ge \s(i+1)$.  Hence $i \in \Des(\s)$, but $i \notin \Exd(\s)$.  It follows that $K(\s) \subseteq \Des(\s)$ and
\bq \label{dexdes}  \Exd(\s) = \left (\Des(\s) \uplus J(\s)\right) - K(\s).\eq
Hence
\bq \label{dexdes2}  |\Exd(\s)| = \des(\s) +| J(\s) | - |K(\s)|.\eq

Let $J(\s) = \{j_1 < j_2 < \dots < j_t\}$ and $K(\s) = \{k_1 < k_2< \dots < k_s\}$.  It is easy to see that if
$\s(1) = 1$ then $t= s$ and
\bq \label{shuf}  j_1 < k_1 <j_2 < k_2 <\dots <j_t <k_t, \eq  since neither $1$ nor $n$ are excedance positions.  On the other hand if $\s(1) \neq 1$ then $s=t+1$ and
 \bq \label{shuf2}  k_1 < j_1 <k_2 <j_2<\dots< j_t < k_{t+1}.\eq  It now follows from (\ref{dexdes2}) that (\ref{exdsize}) holds.

To prove (\ref{exd}), we again handle the cases  $\s(1)=1$ and $\s(1) \ne 1$ separately.

Case 1: Suppose $\s(1) = 1$.  Then (\ref{shuf}) holds.  By (\ref{dexdes}),
$$\sum_{i \in \Exd(\s)} i = \sum_{i \in \Des(\s)}  i - \sum_{i=1}^t (k_i -j_i) . $$
Clearly $$\Exc(\s) = \biguplus_{i = 1}^t \{j_i+1,j_i+2,\dots, k_i\}.$$  Hence
$$\exc(\s) = \sum_{i=1}^t (k_i -j_i) $$ and so (\ref{exd}) holds in this case.

Case 2: Suppose $\s(1) \ne 1$. Now (\ref{shuf2}) holds.
By (\ref{dexdes}),
$$\sum_{i \in \Exd(\s)} i = \sum_{i \in \Des(\s)}  i - k_1 - \sum_{i=1}^t (k_{i+1} -j_{i}) . $$
Now $$\Exc(\s) = \{1,2,\dots, k_1\} \uplus  \biguplus_{i = 1}^t \{j_i+1,j_i+2,\dots, k_{i+1}\},$$
which implies that
$$\exc(\s) =  k_1 + \sum_{i=1}^t (k_{i+1} -j_i) .$$  Hence (\ref{exd}) holds in this case too.
  \end{proof}

We now define the  quasisymmetric functions that  play a central role in this paper. Let  $n\ge 1$, $j,k \ge 0$ and  $\lambda\vdash n$.   The {\em Eulerian quasisymmetric functions}  are defined by $$ Q_{n,j} = Q_{n,j}(\x):=
\sum_{\scriptsize \begin{array}{c} \s \in \sg_n \\ \exc(\s) = j
\end{array}} F_{\Exd(\s),n}(\x).$$
 The {\em fixed point Eulerian quasisymmetric functions} are  defined by $$ Q_{n,j,k} = Q_{n,j,k}(\x):=
\sum_{\scriptsize \begin{array}{c} \s \in \sg_n \\ \exc(\s) = j \\ \fix(\s) =k
\end{array}} F_{\Exd(\s),n}(\x).$$
The {\em cycle type Eulerian quasisymmetric functions} are defined by
$$Q_{\lambda,j} = Q_{\lambda,j}(\x):=
\sum_{\scriptsize \begin{array}{c} \s \in \sg_n \\ \exc(\s) = j \\ \lambda(\s) =\lambda
\end{array}} F_{\Exd(\s),n}(\x),$$
where $\lambda(\s)$ denotes the cycle type of $\s$.
For convenience we  set $[0] =\emptyset$ and $\S_0 = \{\theta\}$,
where $\theta$ denotes the empty word, and let
$\Exd(\theta) = \Des(\theta) = \Exc(\theta) = \emptyset$ and $ \exc(\theta) = \fix(\theta) = \des(\theta) = \maj(\theta)= 0$.   So $Q_{0,0,0} = 1$.  Also set $\lambda(\theta)$ equal to the empty partition of $0$.
 
\begin{remark} It is a consequence of Theorem~\ref{introsymgenth} that the Eulerian quasisymmetric functions $Q_{n,j}$ and $Q_{n,j,k}$ are symmetric functions.  We prove in Section~\ref{symcyclesec} that the refinement $Q_{\lambda,j}$ is symmetric as well.
\end{remark}

Next we define various versions of  $q,p$-analogs of the Eulerian numbers $$a_{n,j}:=|\{\sigma \in \sg_n : \exc(\sigma) =j\}|.$$  Let $n,j,k \ge 0$. The  {\em  $(q,p)$-Eulerian numbers} are defined by $$ a_{n,j}(q,p) =
\sum_{\scriptsize \begin{array}{c} \s \in \sg_n \\ \exc(\s) = j
\end{array}} q^{\maj(\s)} p^{\des(\s)},$$
and the {\em fixed point $(q,p)$-Eulerian numbers} are defined by $$ a_{n,j,k}(q,p) =
\sum_{\scriptsize \begin{array}{c} \s \in \sg_n \\ \exc(\s) = j \\ \fix(\s) =k
\end{array}} q^{\maj(\s)} p^{\des(\s)}.$$
For $n,j \ge 0$ and $\lambda \vdash n$, the {\em cycle type $(q,p)$-Eulerian numbers} are defined by
$$a_{\lambda,j}(q,p) =
\sum_{\scriptsize \begin{array}{c} \s \in \sg_n \\ \exc(\s) = j \\ \lambda(\s) =\lambda
\end{array}}  q^{\maj(\s)} p^{\des(\s)}.$$

It  follows from (\ref{spec1}) and(\ref{exd}) that for $\sigma \in \sg_n$ we have
$$
\Lambda(F_{\Exd(\sigma),n})=(q;q)_n^{-1}q^{\maj(\sigma)-\exc(\sigma)}.
$$
Hence \begin{eqnarray}
 \label{stablespec1} a_{n,j,k}(q,1) &=& q^j (q;q)_n\Lambda (Q_{n,j,k}),\\ \label{stablespec2} a_{\lambda,j}(q,1) &=& q^j (q;q)_n\Lambda (Q_{\lambda,j}).\end{eqnarray}
\begin{proof}[Proof of Corollary~\ref{expgenthfix}]

From equations (\ref{stablespec1}) and (\ref{spech}), we see that Corollary~\ref{expgenthfix} is obtained from Theorem~\ref{introsymgenth} by first replacing $z$ by $z(1-q)$ in  (\ref{introsymgenth1}) and then applying the stable principal specialization $\Lambda$ to both sides of the resulting equation. \end{proof}

The (nonstable) principal specialization is a bit more
complicated. Given two partitions $\lambda \vdash n$ and $\mu
\vdash m$, let $(\lambda,\mu)$ denote the partition of $n+m$
obtained by concatenating $\lambda$ and $\mu$ and then
appropriately rearranging the parts.
\begin{lemma} \label{nonstable} 
For $0 \le j,k \le n$, let $\lambda \vdash n-k$,
where $\lambda$ has no parts of size $1$.  Then
\begin{equation*}a_{(\lambda,1^k),j}(q,p) = (p;q)_{n+1} \sum_{m\ge 0} p^m \sum_{i=0}^k q^{im+j} \Lambda_m (Q_{(\lambda,1^{k-i}), j}).\end{equation*}  Consequently, for $n,j,k \ge 0$,
 \begin{equation*}  a_{n,j,k}(q,p) = (p;q)_{n+1} \sum_{m\ge 0} p^m \sum_{i=0}^k q^{im+j} \Lambda_m (Q_{n-i, j,k-i}).\end{equation*}
\end{lemma}

\begin{proof}

For all $\sigma \in \sg_n$, where $n >0$, we have, by (\ref{spec2}) and Lemma~\ref{exdlem},
$$\sum_{m \ge 0} \Lambda_m(F_{\Exd(\sigma),n}) p^m = {1 \over (p;q)_{n+1} } p^{|\Exd(\s)|+1} q^{\maj(\s) - \exc(\s)}.$$
It follows that \begin{equation} \label{specx} \sum_{m \ge 0} \Lambda_m( Q_{(\lambda,1^k),j}) p^m = X^k_{\lambda,j}(q,p),\end{equation}
where
$$ X^k_{\lambda,j}(q,p) := {1 \over (p;q)_{n+1} }\sum_{\scriptsize \begin{array}{c} \s \in \sg_n \\ \exc(\s) = j \\ \lambda(\s) = (\lambda, 1^k)
\end{array}} q^{\maj(\s)-j } p^{|\Exd(\s)|+1},$$
when $n >0$, and $X_{\lambda,j}^k(q,p) = \frac{1}{1-p}$ when $n=0$.
By (\ref{exdsize}) we have
\begin{eqnarray*} X^k_{\lambda,j}(q,p) = & &{1 \over (p;q)_{n+1} } \sum_{\scriptsize \begin{array}{c} \s \in \sg_n \\ \exc(\s) = j \\ \lambda(\s) = (\lambda, 1^k) \\ \s(1) = 1
\end{array}} q^{\maj(\s)-j } p^{\des(\s)+1} \\ &+&
 {1 \over (p;q)_{n+1} }\sum_{\scriptsize \begin{array}{c} \s \in \sg_n \\ \exc(\s) = j \\ \lambda(\s) = (\lambda, 1^k) \\ \s(1) \ne 1
\end{array}} q^{\maj(\s)-j } p^{\des(\s)}.
\end{eqnarray*}
Let $$ Y^k_{\lambda,j}(q,p) := {1 \over (p;q)_{n+1} }\sum_{\scriptsize \begin{array}{c} \s \in \sg_n \\ \exc(\s) = j \\ \lambda(\s) = (\lambda, 1^k) \\ \s(1) =1
\end{array}} q^{\maj(\s)-j } p^{\des(\s)}.$$
Write  $a^k_{\lambda,j}(q,p)$ for $a_{(\lambda,1^k),j}(q,p)$.
We have  \begin{eqnarray}\label{ayxeq} \nonumber {a^k_{\lambda,j}(q,p) \over q^j(p;q)_{n+1} } &=& Y^k_{\lambda,j}(q,p) + X^k_{\lambda,j}(q,p) - p Y^k_{\lambda,j}(q,p)\\ &=& (1-p) Y^k_{\lambda,j}(q,p)  + X^k_{\lambda,j}(q,p).\end{eqnarray}

Let
$\varphi: \{\s \in \sg_n : \s(1) = 1\} \to \sg_{n-1}$ be the bijection defined by letting $\varphi(\s)$ be the permutation  obtained by removing the $1$ from the beginning of $\s$ and subtracting $1$ from each letter of  the remaining word.
It is clear that  $\maj(\s) = \maj(\varphi(\s)) + \des(\s)$,  $\des(\s) = \des(\varphi(\s)) $, $\exc(\s)=\exc(\varphi(\s)) $ and $\fix(\s)=\fix(\varphi(\s)) +1$.  Hence
$$Y^k_{\lambda,j}(q,p)= {a^{k-1}_{\lambda,j}(q,qp) \over q^j (p;q)_{n+1} }.$$
Plugging this into (\ref{ayxeq}) yields the recurrence
$$ {a^k_{\lambda,j}(q,p)  \over q^j (p;q)_{n+1} }= { a^{k-1}_{\lambda,j}(q,qp) \over q^j(qp;q)_{n} } + X^k_{\lambda,j}(q,p).$$
By iterating this recurrence we get
\begin{equation} \label{thiseq}  {a^k_{\lambda,j}(q,p) \over q^j (p;q)_{n+1} }= \sum_{i=0}^k X^{k-i}_{\lambda,j}(q,q^ip).\end{equation}
The result now follows from (\ref{thiseq}) and (\ref{specx}).
\end{proof}

\begin {proof} [Proof of Corollary \ref{foha}]  By Lemma~\ref{nonstable} and Theorem~\ref{introsymgenth}, we have

 \begin{eqnarray*} \sum_{n\ge 0 }A_{n}^{\maj,\des,\exc,\fix}&&\hspace{-.4in}(q,p,t,r) {z^n \over (p;q)_{n+1}} 
 \\ &=&  \sum_{n,j,k \ge 0} a_{n,j,k}(q,p) t^j r^k {z^n \over (p;q)_{n+1}}\\
 &=&\sum_{n,j,k\ge 0}z^n  t^j r^k \sum_{m\ge 0} p^m \sum_{i=0}^k q^{im+j} \Lambda_m (Q_{n-i, j,k-i})
 \\  &=&  \sum_{m\ge0}p^m  \sum_{i \ge 0} (zrq^m)^i  \sum_{\scriptsize \begin{array} {c}
 n,k\ge i \\ j \ge 0\end{array}} \Lambda_m (Q_{n-i, j,k-i}) (qt)^j r^{k-i} z^{n-i}
  \\ 
 &=& \sum_{m\ge 0}{p^m\over 1-zrq^m}\Lambda_m \left({(1-tq)H(zr) \over H(ztq) -tqH(z)}\right)
\\ &=& \sum_{m \ge 0} p^m {(1-tq) (z;q)_m (ztq;q)_m \over
((z;q)_m -tq(ztq;q)_m)(zr;q)_{m+1}}, \end{eqnarray*}
with the last step following from
$$ \Lambda_m( H(z)) =\Lambda_m\left(\prod_{i \ge 0} \frac 1 {1-x_i z}\right) = \frac 1 {(z;q)_{m}}.$$

\end{proof}

\section{Bicolored necklaces and words} \label{identsec}

In this section we prove Theorem~\ref{introsymgenth}.  There are three main steps.   In the first step (Section~\ref{secnec}) we modify a bijection that  Gessel and Reutenauer \cite{gr} constructed in order to enumerate permutations with a fixed descent set and
fixed cycle type.  This yields an alternative characterization of the Eulerian quasisymmetric functions in terms of  bicolored necklaces.   In the second step (Section~\ref{secban}) we construct  a  bijection from multisets of bicolored necklaces to
bicolored words, which involves Lyndon decompositions of words.  This yields  yet another characterization of the Eulerian quasisymmetric functions.    In the third step (Section~\ref{recsec}) we generalize a bijection that Stembridge constructed to study the representation of the symmetric group on the cohomology of the toric variety associated with the type A Coxeter complex.  This enables us to derive a recurrence relation, which yields
Theorem~\ref{introsymgenth}.

 \subsection{Step 1: bicolored version of the Gessel-Reutenauer bijection} \label{secnec}
The Gessel-Reutenauer bijection (see \cite{gr}) is a   bijection between the set of  pairs $(\sigma,s)$, where $\s$ is a  permutation and $s$ is a ``compatible'' weakly decreasing sequence, and the set of multisets of primitive circular words over the alphabet of positive integers.  This bijection enabled Gessel and Reutenauer  to use the fundamental quasisymmetric functions $F_{\Des(\s),n}$ to study properties of permutations with a fixed descent set and cycle type.   Here we introduce a bicolored version of the Gessel-Reutenauer bijection.  

We  consider circular words over the alphabet of bicolored positive integers $$\mathcal A:=\{1,\bar 1,2,\bar 2, 3,\bar 3, \dots\}.$$   (We can think of ``barred'' and ``unbarred'' as  colors assigned to each positive integer.)  For each such circular
word and each starting position, one gets a linear word by
reading the circular word in a clockwise direction. If one gets a
distinct linear word for each starting position then the
circular word is said to be {\em primitive}.  For example the circular word $(\bar 1,1,1) $  is primitive while the circular word $(\bar 1, 2,\bar 1, 2)$  is not.  (If $w$ is a linear word then $(w)$ denotes the circular word obtained by placing the letters of $w$ around a circle in a clockwise direction.)  The
{\em absolute value} or just {\em value} of a letter $a$ is the letter obtained by ignoring
the bar if there is one. We denote this by $|a|$. The {\em size} of a circular word is the number of letters in the word (counting multiplicity).  
\begin{defn}  \label{orndef} A  {\em bicolored  necklace} is a primitive circular word $w$ over alphabet $\mathcal A$   such that if the size of $ w$ is greater than $1$ then \begin{enumerate}
\item every barred letter  is followed
(clockwise) by a letter less  than or equal to it  in absolute value
\item every unbarred  letter is followed by a letter greater than or equal to it  in absolute value.
\end{enumerate}      A circular
word of size $1$ is a bicolored necklace  if its sole letter is unbarred.
 A
{\em bicolored ornament}  is a
multiset
of
bicolored necklaces.
\end{defn}

For example the following
circular words are bicolored necklaces:
$$ (\bar 3, 1 ,3, \bar 3, 2, 2), (\bar 3,1, \bar 3, \bar 3, 2, 2) , (\bar 3, 1, 3, \bar 3, \bar 2, 2),
(\bar 3, 1, \bar 3, \bar 3, \bar 2, 2), (2),$$
 while $(\bar 3, \bar 1, 3, 3, 2, \bar 2) $
 and $(\bar 3)$  are not.

From now on we will drop the word ``bicolored'';   so ``necklace'' will stand for  ``bicolored necklace'' and ``ornament"  will stand for  ``bicolored ornament''.  The type $\lambda(R)$ of an ornament $R$ is the
partition whose parts are the sizes of the necklaces  in $R$. The weight of a letter $a$ is  the
indeterminate $x_{|a|}$.  The
weight $\wt(R)$ of an ornament $R$  is the product of the weights of
the letters of  $R$.  For example
$$\lambda((\bar 3,2,2),(\bar 3,\bar 2, 1,1,2)) = (5,3)$$
 and
$$\wt((\bar 3,2,2),(\bar 3,\bar 2,1,1,2)) = x_3^2x_2^4x_1^2.$$
For each partition $\lambda$ and nonnegative integer $j$, let
$\mathfrak R_{\lambda,j}$ be the set of ornaments of type
$\lambda$ with $j$ bars.

Given a permutation  $\sigma \in \S_n$, we say that a weakly decreasing sequence $s_1\ge s_2\ge \dots \ge s_n$ of
positive integers is  $\sigma$-{\em compatible} if
$s_i > s_{i+1}$ whenever $i \in \Exd(\sigma)$.  For example,
$7,7,7,5,5,4,2,2$ is $\s$-compatible, where $\s=45162387$.
Note that for all $\s \in \mathfrak S_n$,
$$F_{\Exd(\s),n} = \sum_{s_1,\dots, s_n} x_{s_1}\dots x_{s_n},$$
where $s_1,\dots,s_n$ ranges over all $\s$-compatible sequences.

 For $\lambda \vdash n$ and $j =0,\dots,n-1$, let $\Com_{\lambda,j}$ be the set of pairs $(\s,s)$, where $\s $ is a permutation of cycle type $\lambda$ with $j$ excedances and $s$ is  a $\s$-compatible sequence.
 Let  $\phi: \Com_{\lambda,j} \to {\mathfrak R}_{\lambda,j} $ be the map defined by letting
$\phi(\sigma, s)$ be the ornament obtained by first writing $\sigma$ in cycle form with bars above the letters that are followed (cyclicly) by larger letters (i.e., the excedances) and then replacing each $i$ with $s_i$, keeping the bars in place.  For example,
let $\s = 45162387$ and $s=7,7,7,5,5,4,2,2$.  First we write $\s$  in cycle form,  $$\s = (1,4,6,3)(2,5) (7,8).$$
Next we put  bars above the
letters that are followed  by larger letters, $$(\bar 1,\bar 4,6,3)(\bar 2,5) (\bar 7,8).$$
After replacing each $i$ by $s_i$, we have the ornament $$(\bar 7,\bar 5,4,7)(\bar 7,5) (\bar 2,2).$$

\begin{thm} \label{ornbij} The map $\phi: \Com_{\lambda,j} \to {\mathfrak R}_{\lambda,j} $ is a  well-defined bijection.
\end{thm}

\begin{proof}

Let $(\s,s) \in \Com_{\lambda,j}$.  It is clear that the circular words in the multiset $\phi(\sigma,s)$ satisfy conditions (1) and (2) of Definition~\ref{orndef} and that  the number of bars of $\phi(\sigma,s)$ is $j$.  It is also clear that $\lambda(\phi(\s,s))= \lambda$.

We must now show that the circular words in $\phi(\sigma,s)$ are primitive.
Suppose there is a circular word $(a_{i_1}, \dots, a_{i_k})$ that is not primitive.  Let this word come from the cycle $(i_1,i_2,\dots,i_k)$ of $\sigma$, where $i_1$ is the smallest element of the cycle.   Suppose $a_{i_1}, \dots, a_{i_k} = (a_{i_1},\dots a_{i_d})^{k/d}$, where $d < k$.   Since $a_{i_1} = a_{i_{d+1}}$,  $i_1 < i_{d+1}$,  and the $|a_i|$'s form a weakly decreasing sequence, it follows that $|a_{i}| = |a_{i_1}|$ for all $i$ such that $i_1 \le  i  \le i_{d+1}$.  Hence since $(|a_1|, |a_2|, \dots, |a_n| )$ is $\s$-compatible, there can be no element of $\Exd(\s)$ in the set $\{i_1,i_1+1, \dots,i_{d+1}-1\}$. Since $a_{i_1} $ and $a_{i_{d+1}}$ are equal, they are both barred or both unbarred. (Actually, $a_{i_1}$ is barred since $i_1$ is the smallest element of its cycle, but since we repeat this argument for $i_2$ in the next paragraph, we don't want to use this fact.)  It follows that both $i_1$ and $i_{d+1}$ are  excedance positions or  both are  not.   Hence since no element of $\Exd(\s)$ is in $\{i_1,i_1+1, \dots,i_{d+1}-1\}$, we have   $$\s(i_1) < \s(i_1+1) < \dots < \s(i_{d+1}).$$

Since $\s({i_1}) = i_2$ and $\s(i_{d+1})= i_{d+2}$ we have $i_2 < i_{d+2}$.  Repeated use of the above argument yields  $i_3 < i_{d+3}$ and eventually  $i_{k+1-d} < i_{1}$, which is impossible since $i_1$ is the smallest element of its cycle.  Hence all the circular words of $\phi(\sigma,s)$ are primitive and $\phi(\sigma,s)$ is an ornament of type $\lambda$ with $j$ bars, which means that $\phi$ is  well-defined.

To show that $\phi$ is a bijection, we construct its inverse $\eta:{\mathfrak R}_{\lambda,j} \to \Com_{\lambda,j}$, which takes the ornament $R$ to the pair $(\s(R),s(R))$.
Here $s(R)$ is simply  the weakly decreasing rearrangement of the letters  of $R$ with bars removed.  To construct the  permutation $\s(R)$, we need to first order the alphabet  $\mathcal A$  by
\begin{eqnarray} \label{alphorder} 1<\bar 1 < 2 <  \bar 2 < \dots .\end{eqnarray}   Note that this ordering is different from the ordering  (\ref{order1}) of the finite
alphabet that was used to define $\Exd$.

For each position $x$ of each necklace  of $R$, consider the infinite word
$w_x$ obtained by reading the necklace in a clockwise direction starting at position $x$.      Let $$w_x \le_L w_y$$ mean that $w_x$ is lexicographically less than or equal to $w_y$.  We use the lexicographic order on words  to order the positions:  if $w_x <_L w_y$ then we say  $x< y$.  We break ties as follows:  if   $w_x= w_y$ and $x \ne y$ then $x$ and $y$ must be positions in distinct (but equal) necklaces since necklaces are primitive.  If  $w_x= w_y$   and  $x$ is a position in an earlier necklace than that of $y$ under some fixed linear ordering of the necklaces   then let $x<y$.  With this ordering on words in hand, if $x$ is the ith {\em largest } position  then replace the letter in position $x$ by $i$.  We now have a multiset of circular words in which each letter  $1,2,\dots,n$ appears exactly once.  This is the cycle form of the permutation $\s(R)$.

For example, if
$$R= ((\bar 7,  \bar 3,  3, 5), (\bar 7, 3, \bar 5, 3), (\bar 7, 3, \bar 5, 3), (5)),$$
then $$\s(R) = (1,8,13,6), (2,11,4,9) ,(3,12,5,10),(7) $$and $$ s(R) =  7,7,7,5,5,5,5,3,3,3,3,3,3.$$

It is not hard to see  that the letter in position $x$ of $R$ is barred if and only if  the letter that replaces it is  an excedance position of $\s(R)$.  This implies  that the number of bars of $R$ equals the number of excedances of $\s(R)$.

Now we  show that $s(R)$ is $\s(R)$-compatible.
Suppose $$s(R)_i = s(R)_{i+1}.$$ We must show $i \notin \Exd(\s(R))$.  Let $x$ be the $i$th largest position of $R$ and let $y$ be the $(i+1)$st largest.  Then $i$ is placed in position $x$ and $i+1$ is placed in position $y$.   Let $f(w)$ denote the first letter of a word $w$.  Then one of the following must hold
\begin{enumerate}
\item
$f(w_x) =  s(R)_i = f(w_y)$
 \item $f(w_x) =  \overline{s(R)_i}= f(w_y)$
\item $f(w_x) =  \overline{s(R)_i}$
and  $f(w_y) = s(R)_i$ .
\end{enumerate}

Cases (1) and (2):    Either $w_x  >_L w_y$ or  $w_x = w_y$ and $x$ is in an earlier necklace than $y$. Let $u$ be the position that follows $x$ clockwise and let $v$ be the position that follows $y$.
  Since  $w_u$ is the word obtained from $w_x$ by removing its first letter,  $w_v$ is obtained from $w_y$ by removing its first letter, and the first letters are equal, we conclude that   $u >v$.  Hence the letter that gets placed in position $u$ is smaller than the letter that gets placed in position $v$.  Since the letter  placed in position $u$ is  $\s(R)(i) $ and the letter placed in position $v$ is $\s(R)(i+1) $, we have $\s(R)(i) <\s(R)(i+1) $. Since  $i,i+1 \in \Exc(\s(R))$ or $i,i+1 \notin \Exc(\s(R))$, we conclude that $i \notin \Exd(\s(R))$.

Case (3):  Since the letter in position $x$ is barred we have $i \in \Exc(\s(R))$.  Since the letter in position $y$ is not barred, we have $i+1 \notin \Exc(\s(R))$.  Hence $i \notin \Exd(\s(R))$.

In all three cases we have that $s(R)_i = s(R)_{i+1}$ implies $i \notin \Exd(\s(R))$.  Hence $s(R)$ is $\s(R)$-compatible.

Now we show that the map $\eta$  is the inverse of $\phi$.  It is easy to see  that $\phi(\eta(R)) =R$.
Establishing $\eta(\phi(\s,s)) = (\s,s)$ means establishing $\s(\phi(\s,s)) = \s$ and  $s(\phi(\s,s)) = s$.  The latter equation is obvious.   To establish the former, let $R= \phi(\sigma,s)$.  Recall that $R$ is obtained by writing $\s$ in cycle form, barring the excedances,  and then replacing each  $i$ by $s_i$, keeping the bars intact.    Let $p_i$ be the position that $i$ occupied before the replacement.  By ordering the cycles of $\s$ so that the minimum elements of the  cycles increase, we get  an ordering of the necklaces in $R$, which we use to break ties between the $w_{p}$. To show that $\s(R) = \s$, we need
to show that  \begin{enumerate}
\item if $i<j$ then $w_{p_i}   \ge_L w_{p_j}$
\item  if $i<j$ and $w_{p_i}  = w_{p_j} $ then $i$ is in a cycle whose minimum element is less than that of the cycle containing $j$.
\end{enumerate}

To prove (1) and (2),  we will use the following implication:
\begin{equation} \label{claimimp} i<j \mbox{ and }w_{p_i} \le_L w_{p_j}  \implies f(w_{p_i}) = f(w_{p_j}) \mbox{ and } \s(i) < \s(j) ,\end{equation} (Recall $f(w) $ is the first letter of a word $w$.)

Proof of (\ref{claimimp}):  Since $i <j$ and $s$ is weakly decreasing, we have  $s_i \ge s_j$.    Since $w_{p_i} \le_L w_{p_j}$, we have $f(w_{p_i}) \le f(w_{p_{j}})$, which implies that  $s_i \le s_j$.   Hence $s_i=s_j$, which implies $$s_i = s_{i+1} = \dots = s_j.$$  It follows from this and  the fact that $s$ is $\s$-compatible that $k \notin \Exd(\s)$ for all $k= i, i+1, \dots, j-1$.   This implies that either $$i,i+1,\dots, j \in \Exc(\s) \mbox{ and  } \s(i) < \s(i+1) <\dots < \s(j)$$ or  $$i,i+1,\dots, j \notin \Exc(\s) \mbox{ and } \s(i) < \s(i+1) <\dots < \s(j)$$  or $$i \in \Exc(\s) \mbox { and }  j \notin \Exc(\s).$$ In the first case, $f(w_{p_i}) =\bar s_i =\bar s_j = f(w_{p_j})$.  In the second case,  $f(w_{p_i}) = s_i =s_j = f(w_{p_j})$.  In the third case $f(w_{p_i}) = \bar s_i >s_i = s_j = f(w_{p_j})$,  which is impossible.   Hence the conclusion of the implication (\ref{claimimp}) holds.

Now we use (\ref{claimimp}) to prove (1).  Suppose $i<j$ and $w_{p_i} <_L w_{p_j}$.  Then by (\ref{claimimp}), we have $\s(i) < \s(j)$ and  $f(w_{p_i}) = f(w_{p_j})$, which implies $w_{p_{\s(i)}} <_L w_{p_{\s(j)}}$.  Hence we can apply (\ref{claimimp}) again with $\s(i)$ and $\s(j)$ playing the  roles of $i$ and $j$.  This yields  $f(w_{p_{\s^2(i)}}) = f(w_{p_{\s^2(j)}})$,  $\s^2(i) <\s^2(j)$, and $w_{p_{\s^2(i)}} <_L w_{p_{\s^2(j)}}$.    Repeated application of (\ref{claimimp}) yields
$ f(w_{p_{\s^m(i)}}) = f(w_{p_{\s^m(j)}})$ for all $m$, which implies $w_{p_i} =w_{p_j}$, a contradiction.

Next we prove (2).  Repeated application of (\ref{claimimp}) yields $\s^m(i) < \s^m(j)$ for all $m$.  Hence the cycle containing $i$ has a smaller minimum than the cycle containing $j$.
\end{proof}

\begin{cor} \label{ornth}
For all $\lambda\vdash n$ and $j= 0,1,\dots,n-1$,
$$Q_{\lambda,j} = \sum_{R \in \mathfrak R_{\lambda,j}} \wt(R).$$
\end{cor}

Corollary~\ref{ornth}  has several interesting consequences.  For one thing,
it can be  used  to prove that the Eulerian quasisymmetric functions
$Q_{\lambda,j}$ are actually symmetric (see Section~\ref{secQ}).  It also
has the following useful consequence.

\begin{cor}
\label{dercor} For all $n,j,k$,  $$ Q_{n,j,k} =  h_k
Q_{n-k,j,0}.$$
\end{cor}

It follows from Corollary~\ref{dercor} that Theorem~\ref{introsymgenth}
is equivalent to
\begin{equation*} \label{symgen2}
\sum_{n,j \ge 0} Q_{n,j,0} t^j z^n = \frac{1}{1-\sum_{n \ge 2} t [n-1]_t h_n z^n}.
\end{equation*}
We rewrite this as,
$$1 = (\sum_{n,j \ge 0} Q_{n,j,0} t^j z^n) ( 1-\sum_{n \ge 2} t [n-1]_t h_n z^n),$$
which is equivalent to
$$\sum_{n,j \ge 0} Q_{n,j,0} t^j z^n = 1+ \sum_{n\ge 2}\sum_{\scriptsize\begin{array}{c} j\ge 0
\\ 0 \le m \le n-2\end{array}} \!\!\!Q_{m,j,0} t^{j+1} [n-m-1]_t h_{n-m}\,\, z^n.$$
For $n \ge 2$, we have
\begin{eqnarray*}\sum_{\scriptsize\begin{array}{c} j\ge 0; \\ 0 \le m \le n-2\end{array}} \!\!\!\!Q_{m,j,0} t^{j+1} [n-m-1]_t h_{n-m} &=& \sum_{\scriptsize\begin{array}{c} j\ge 0 \\ 0 \le m \le n-2\end{array}}\! \!\!\!Q_{m,j,0} h_{n-m} \sum_{i=j+1}^{j+n-m-1} t^i \\ &=& \sum_{i \ge 0} t^i \!\!\!\!\sum_{\scriptsize\begin{array}{c}0 \le m \le n-2\\ i-n+m < j <i \end{array}}\! \!\!\!Q_{m,j,0} h_{n-m}
\end{eqnarray*}
Hence Theorem~\ref{introsymgenth} is equivalent to the recurrence relation
\begin{equation} \label{rr}
Q_{n,j,0} = \sum_{\scriptsize \begin{array}{c}0 \le m \le n-2
\\ j+m-n < i < j \end{array}} Q_{m,i,0}   h_{n-m}, \end{equation}
for $n \ge 2$ and $j \ge 0$.

\subsection{Step 2: banners} \label{secban}

In order to establish the recurrence relation (\ref{rr}), we  introduce another type
of configuration, closely related to ornaments.

\begin{defn}  \label{bandef} A  {\em banner} is a word $B$ over alphabet $\mathcal A$    such that 
the last letter of $B$ is unbarred and 
for all $i =1,\dots ,\ell(B)-1$, \begin{enumerate}
\item if $B(i)$ is barred then $|B(i)| \ge |B(i+1)|$
\item if $B(i)$ is unbarred then $|B(i)| \le |B(i+1)|$,
\end{enumerate}
where $B(i)$ denotes the $i$th letter of $B$ and $\ell(B)$ denotes the length of $B$. 
\end{defn}

A {\em Lyndon word} over an ordered alphabet is a  word that is
strictly lexicographically larger than all its circular rearrangements.  A
{\em Lyndon factorization} of a word over an ordered alphabet is a
factorization into a weakly lexicographically  increasing sequence
of  Lyndon words.  It is a result of Lyndon (see  \cite[Theorem~5.1.5]{p}) that every
word has a unique Lyndon factorization.  The {\em Lyndon type} $\lambda(w)$ of a word $w$
is the partition whose parts are the lengths of the words in its
Lyndon factorization.

To apply the theory of Lyndon words to banners, we use the ordering of $\mathcal A$ given in (\ref{alphorder}).  Using this order,  the  banner $B:= \bar 2 2\bar 7 5\bar 7\bar 547 $  has Lyndon factorization   $$\bar 2 2\cdot \bar 7 5 \cdot\bar 7\bar 547 .$$ So the Lyndon type  of $B $ is the partition $(4,2,2)$.

The weight $\wt(B)$ of a banner $B$ is the product of the weights of
its letters, where as before the weight of a letter $a$ is $x_{|a|}$.  For each partition $\lambda$ and nonnegative
integer $j$, let $\mathfrak B_{\lambda,j}$ be the set of banners
with $j$ bars whose Lyndon type is $\lambda$.

\begin{thm} \label{banprop} For each partition $\lambda$ and nonnegative
integer $j$, there is a weight-preserving bijection
$$\psi: \mathfrak B_{\lambda,j} \to \mathfrak R_{\lambda,j}.$$
Consequently, $$Q_{\lambda,j} = \sum_{B \in  \mathfrak B_{\lambda,j}} \wt(B).$$ \end{thm}

\begin{proof} First note that there is a natural weight-preserving bijection from the set of Lyndon banners  of length $n$ to the set of necklaces of size $n$.  To go from a Lyndon banner $B$ to a necklace $(B)$ simply  attach  the ends of  $B$ so that the left end follows the right end when read in a clockwise direction.    To go from a necklace back to a banner simply find the lexicographically largest linear word obtained by reading the circular word in a clockwise direction.
The number of bars of $B$  clearly is the  same as that of $(B)$.

Let $B \in \mathfrak B_{\lambda,j} $ and let
$$B = B_1 \cdot B_2 \cdots B_k$$ be the unique Lyndon factorization of $B$.  Note that each $B_i$ is a  Lyndon  banner.  To see this we need only check that the last letter of each $B_i$ is unbarred.  The last letter of $B_k$ is the last letter of $B$; so it is clearly unbarred.  For $i < k$, the  last letter of $B_i$ is strictly less than the first letter of $B_i$, which is less than or equal to the first letter of $B_{i+1}$.  Since the last letter of $B_i$ immediately precedes the first letter of $B_{i+1}$ in the banner $B$,  it must be unbarred.

 Now define
$\psi(B)$ to be the ornament whose necklaces are $$(B_1), (B_2), \dots, (B_k).$$ This map is clearly weight preserving, type preserving, and bar preserving.   To go from an ornament back to a banner simply arrange the Lyndon banners obtained from the necklaces in the ornament in weakly increasing lexicographical order and then concatenate.
\end {proof}

\subsection{Step 3: the recurrence relation} \label{recsec}
Now we prove the recurrence relation (\ref{rr}), which we have shown is equivalent to Theorem~\ref{introsymgenth}.  It is convenient to 
 rewrite (\ref{rr}) as
\begin{equation} \label{rr2}
Q_{n,j,0} = \sum_{\scriptsize \begin{array}{c}0 \le m \le n-2
\\1\le b < n-m \end{array}} Q_{m,j-b,0} \,  h_{n-m}.\end{equation}

Define a {\em marked sequence} $(\omega,b)$  to be a weakly
increasing finite sequence $\omega$ of positive integers together
with an integer  $b$ such that $1 \le b < \mbox
{length}(\omega)$.  (One can visualize this as a weakly increasing sequence with a mark above one of its elements other than  the last.) For $n \ge 2$, let $\mathfrak M_n$ be the set of marked
sequences of length $n$.  For $n \ge 0$,  let ${\mathfrak B}^0_n$ be the set
of banners of length $n$ whose Lyndon type has no parts of size
$1$.    It will be convenient to consider the empty word to be a banner of length $0$, weight $1$, with no bars, and whose Lyndon type is the partition of $0$ with no parts.  So ${\mathfrak B}^0_0$  consists of a single element,  namely the empty word.  Note that ${\mathfrak B}^0_1$ is the empty set.

Theorem~\ref{banprop} and
Theorem~\ref{bij} below suffice to establish the
recurrence relation (\ref{rr2}).  Indeed, by Theorem~~\ref{banprop} the monomial terms on the left of (\ref{rr2}) are the weights of banners in $\mathfrak B_n^{0}$  with $j$ barred letters. By Theorem~\ref{bij},
these banners correspond bijectively to pairs $(B^\prime,(\omega,b))$ such that $B^\prime\in\mathfrak B_m^0$ has $j-b$ barred letters, $( \omega,b) \in \mathfrak M_{n-m}$, and $\wt(B^\prime)\wt(\omega)$ equals the weight of the corresponding banner in $\mathfrak B_n^{0}$.   Since the monomial terms of $h_{n-m}$ are the weights of weakly increasing sequences of length $n-m$, we have that the monomial terms of $Q_{m,j-b,0}h_{n-m}$ in the sum on the right of  (\ref{rr2})  are also of the form
$\wt(B^\prime)\wt(\omega)$, where $B^\prime\in\mathfrak B_m^0$ has $j-b$ barred letters and $( \omega,b) \in \mathfrak M_{n-m}$.

\begin{thm} \label{bij} For all $n \ge 2 $, there is a bijection
$$\gamma:  {\mathfrak B}^0_{n} \to \biguplus_{0 \le m\le n-2}
{\mathfrak B}^0_{m} \times \mathfrak M_{n-m},$$ such that  if
$\gamma(B) = (B^\prime,(\omega,b)) $ then \bq \label{wteq} \wt(B) =
\wt(B^\prime)\wt(\omega)\eq and \bq \label{bbbeq} \bbar(B) = \bbar(B^\prime) + b,\eq
where $\bbar(B)$ denotes the number of bars of $B$. \end{thm}

We will make use of another type of factorization of a
word over an ordered alphabet
used by D\'esarm\'enien and Wachs \cite{dw}.   For any alphabet $A$, let
$A^+$ denote the set of words over $A$ of positive length.

\begin{defn}[\cite{dw}]  An {\em increasing} factorization of a word $w$ of positive length over a
totally ordered alphabet $A$ is a factorization
$w=w_1 \cdot w_2\cdots  w_k$ such that
\begin{enumerate}
\item each $w_i$ is of the form $a_i^{j_i} u_i$, where $a_i \in A$, $j_i > 0$ and $$u_i \in \{x \in A : x < a_i\}^+$$
\item $a_1 \le a_2 \le \dots \le a_k$.
\end{enumerate}
\end{defn}

For example, $87\cdot 8866\cdot 995587 \cdot 95$ is an increasing factorization of
the word $w:= 87886699558795$ over the totally ordered alphabet of positive integers.
 Note that this factorization is different from the Lyndon factorization of $w$, which is   $87\cdot 8866\cdot 99558795$

 \begin{prop}[{\cite[Lemmas 3.1 and  4.3 ]{dw}}]  \label{deswa} A word over an ordered alphabet admits an increasing factorization if and only if its Lyndon type has no parts of size 1.   Moreover,  the increasing factorization is unique.
 \end{prop}

\begin{proof}[Proof of Theorem~\ref{bij}]     Given a banner $B$ in $ {\mathfrak B}^0_{n} $, take its unique increasing factorization $$B=B_1 \cdot B_2 \cdots B_k,$$ whose existence is guaranteed  by Proposition~\ref{deswa}.    We will extract an increasing word from $B_k$.
We have
$$B_k= a^p i_1\cdots i_l,$$
where $p,l \ge 1$, and $a > i_1,i_2,\dots,i_l.$  For convenience set $i_0 = a$.   It follows from the definition of banner that $a$ is a barred letter.      Determine the unique index $r$ such that $i_1\ge \dots \ge i_{r-1}$ are barred, while $i_{r} $ is unbarred.     Then let $s $ be  the unique index that satisfies  $r \le s \le l$ and either
 \begin{enumerate}  \item   $i_r \le i_{r+1} \le \dots \le i_{s}$ are all unbarred and less than $i_{r-1}$ (note that $i_s$ can be equal to $i_{r-1}$ in absolute value), while $i_{s+1} > i_{r-1}$ if $s <l$, or
 \item $i_r \le i_{r+1} \le \dots \le i_{s-1}$  are all unbarred and less  than  $i_{r-1}$, and $i_s$ is barred and  less than or equal to $i_{r-1}$.
 \end{enumerate}

  {\bf Case 1: } $s= l$.  In this case (1) must hold since the last letter of a banner is unbarred.  Let $\omega$ be the weakly increasing rearrangement of $B_k$ with bars removed and let
\bq \label{case1} B^\prime = B_1 \cdot B_2 \cdots B_{k-1}.\eq
  To see that $B^\prime$ is a banner, one need only note that the last letter of $B_{k-1}$ is unbarred (as is the last letter of each $B_i$).  Since (\ref{case1}) is an increasing factorization of $B^\prime$, it follows from Proposition~\ref{deswa} that the Lyndon type of $B^\prime$ has no parts of size 1.  Let $$\gamma(B) = (B^\prime, (\omega,b)),$$ where $b$ is the number of bars of $B_k$.   Clearly $p \le b < p+l $; so $(\omega,b) $
  is a marked sequence for which (\ref{wteq}) and (\ref{bbbeq}) hold.
  For example, if
  $$B=\bar 2\bar 2\bar 2 1 \cdot \bar 5 22 \bar 4 2\cdot  \bar 8\bar 8 \bar 7 \bar 5 2235.$$
  then $a=\bar 8$, $p=2$, $ r= 3$ and $s= 6 = l$.  It follows that $$(\omega, b) = (22355788, 4) \mbox{ and }
  B^\prime = \bar 2\bar 2\bar 2 1 \cdot \bar 5 22 \bar 4 2.$$

{\bf  Case 2: }  $s < l$. In this case either (1) or (2) can hold.  Let $b$ be the number  of  bars of $i_1,i_2, \dots, i_s$.  Clearly $b <s$.   If (1) holds then $r>1$ since $i_{s+1} > i_{r-1}$; so $b>0$.   If (2) holds  clearly $b > 0$.   Let $\omega$ be the increasing rearrangement of $i_1,i_2, \dots, i_s$ with bars removed,  let $B_k^\prime =
 a^p i_{s+1} \cdots i_{l}$, let
 \bq \label{case2} B^\prime =  B_1 \cdot B_2 \cdots B_{k-1} \cdot B^\prime_k ,\eq
  and let $$\gamma(B) = (B^\prime, (\omega,b)).$$  Clearly  $B^\prime $ is a banner with increasing factorization given by (\ref{case2})  and $(\omega,b) $
 is a marked sequence for which (\ref{wteq}) and (\ref{bbbeq}) hold.
  For example, if
  $$B=\bar 2\bar 2\bar 2 1 \cdot \bar 5 22 \bar 4 2\cdot  \bar 8\bar 8 \bar 7 \bar 5 2235\bar 6 24$$ then $a=\bar 8$, $p=2$, $ r= 3$, and $s = 6 < l$.  Hence
  $$(\omega, b) = (223557, 2) \mbox{ and }
  B^\prime = \bar 2\bar 2\bar 2 1 \cdot \bar 5 22 \bar 4 2 \cdot \bar 8 \bar 8 \bar 6 24.$$
   If
  $$B=\bar 2\bar 2\bar 2 1 \cdot \bar 5 22 \bar 4 2\cdot  \bar 8\bar 8 \bar 7 \bar 5 223\bar 5 4\bar 6 24$$ then $a=\bar 8$, $p=2$, $ r= 3$, and $s = 6 < l$.  Hence
  $$(\omega, b) = (223557, 3) \mbox{ and }
  B^\prime = \bar 2\bar 2\bar 2 1 \cdot \bar 5 22 \bar 4 2 \cdot \bar 8 \bar 8 4\bar 6 24.$$

 In order to prove that the map $\gamma$  is a bijection, we describe its inverse.
 Let  $ (B^\prime, (\omega,b) ) \in {\mathfrak B}^0_{m} \times \mathfrak M_{n-m}$, where
$0 \le m\le n-2$.
Let
  $$B^\prime = B_1 \cdot B_2 \cdots B_{k-1}$$
  be the unique increasing factorization of $B^\prime$,  whose existence is guaranteed by Proposition~\ref{deswa}, and
  let $a$ be the largest letter of $B_{k-1}$.  We also let $\omega_i$ denote the $i$th letter of $\omega$.

 {\bf Case 1: } $|a| \le \omega_{n-m}$.  Let
 $$B_{k} = \bar{\omega}_{n-m} \cdots \bar{\omega}_{n-m-b+1} \omega_1 \cdots \omega_{n-m-b}.$$
 Clearly $B_{k}$ is a  banner with exactly $b$ bars that are placed on a rearrangement of $\omega$.  Now let
 $$B = B_1 \cdot B_2 \cdots B_{k-1} \cdot B_{k}.$$  It is easy to see that this is an increasing decomposition of a banner and that equations (\ref{wteq}) and (\ref{bbbeq}) hold.

 {\bf Case 2: } $|a| > \omega_{n-m}$.  In this case we expand the banner $B_{k-1}$ by inserting the letters of $\omega$ in the following way.  Suppose
 $$B_{k-1} = a^p j_1\cdots j_l,$$
  where $p,l \ge 1$, and $a > j_i$ for all $i$.
 If $j_1> \bar{\omega}_{n-m-b+1}$  let $$\tilde B_{k-1} = a^p  \bar{\omega}_{n-m} \cdots \bar{\omega}_{n-m-b+1} \omega_1 \cdots \omega_{n-m-b}  j_1,\dots,j_l.$$  Otherwise if $j_1\le \bar \omega_{n-m-b+1}$ let
 \beq\tilde B_{k-1} = a^p  \bar{\omega}_{n-m} \cdots \bar{\omega}_{n-m-b+2} \omega_1 \cdots\omega_{n-m-b}\ \bar{\omega}_{n-m-b+1}  j_1,\dots,j_l.\eeq
 In both cases  $\tilde B_{k-1}$ is a banner.  Now let
 $$B = B_1 \cdot B_2 \cdots  B_{k-2} \cdot \tilde B_{k-1}.$$  It is easy to see that this is an increasing decomposition of a banner and that equations (\ref{wteq}) and (\ref{bbbeq}) hold.
It is also easy to check that the map $( B^\prime, (\omega,b)) \mapsto B$ is the inverse of $\gamma$.
\end{proof}

\begin{remark} The bijection $\gamma$ when restricted to banners with distinct letters (permutations) reduces to   a
bijection that Stembridge
\cite{stem1} constructed to study the representation of the
symmetric group on the cohomology of the toric variety associated
with the type A Coxeter complex (see Section \ref{repthsec}).    For words with distinct letters, the notions of  decreasing decomposition and Lyndon decomposition  coincide.   In \cite{stem1} the Lyndon decomposition corresponds to the cycle decomposition of a permutation, and marked sequences are defined differently there.  
There is, however,  a close connection between our notion of marked sequences and Stembridge's.
\end{remark}

\begin{remark} In Section~\ref{othersec} we discuss a connection, pointed out to us by Richard Stanley,  between banners and   words with no adjacent repeats.  This connection can be used to provide an alternative to Step 3 in our proof of Theorem~\ref{introsymgenth}.\end{remark}

\section{Alternative formulations} \label{altsec}
In this section we present some equivalent formulations of 
Theorem~\ref{introsymgenth} and some immediate consequences.

\begin{cor}[of Theorem~\ref{introsymgenth}]  Let $Q_n(t,r) = \sum_{j,k \ge 0} Q_{n,j,k}\, t^j \,r^k$. Then $Q_n(t,r)$ satisfies the following recurrence relation:
\begin{equation} \label{altrecrel} Q_n(t,r) = r^nh_n + \sum_{k=0}^{n-2} Q_k(t,r) h_{n-k} t [n-k-1]_t. \end{equation}
\end{cor}

\begin{proof}
The recurrence relation is equivalent to
$$\sum_{k=0}^{n} Q_k(t,r) h_{n-k} t [n-k-1]_t = -r^n h_n,$$ 
where $[-1]_t:= -t^{-1}$. Taking the generating function
we have
$$\left(\sum_{n\ge 0} Q_n(t,r) z^n \right ) \left( \sum_{n\ge 0}h_n t[n-1]_t z^n \right ) =- H(rz).$$  The result follows from this.
\end{proof}

The right hand side of (\ref{introsymgenth1}) is the Frobenius characteristic of a graded permutation representation that Stembridge \cite{stem1} described in terms of $\sg_n$ acting on ``marked words".
From his work we were led to the following formula, which can easily be proved by showing that the right hand side satisfies the recurrence relation (\ref{altrecrel}).

 \begin{cor} \label{formQcor} For all $n \ge 0$, \begin{equation} \label{formQ} Q_n(t,r)=\sum_{m = 0}^{\lfloor {n \over 2} \rfloor}\,\, \!\!\!\!\sum_{\scriptsize
\begin{array}{c} k_0\ge 0  \\ k_1,\dots, k_m \ge 2 \\ \sum k_i = n
\end{array}}
\!\!\!\!\!\!\! r^{k_0}h_{k_0}
\prod_{i=1}^m h_{k_i} t [k_i-1]_{t}.\end{equation}
 \end{cor}

Let $$\left[\begin{array}{c} n \\k\end{array}\right]_q = {[n]_q! \over [k]_q! [n-k]_q!}\, \mbox{  and  } \,\left[\begin{array}{c} n \\k_0,\dots,k_m\end{array}\right]_q = {[n]_q!
\over [k_0]_q![k_1]_q!\cdots [k_m]_q!}.$$ By taking the principal stable specialization of both sides of  the recurrence relation (\ref{altrecrel}) and the formula (\ref{formQ}), we have the following result.

\begin{cor} For all $n \ge 0$,
$$ A_n^{\maj,\exc, \fix}(q,t,r) = r^n + \sum_{k=0}^{n-2} \left[\begin{array}{c} n \\k\end{array}\right]_q\,\,
A_k^{\maj,\exc, \fix}(q,t,r) \,tq[n-k-1]_{tq},$$
and
$$A_n^{\maj,\exc, \fix}(q,t,r)  =   \sum_{m = 0}^{\lfloor {n \over 2} \rfloor}  \!\!\!\!\sum_{\scriptsize
\begin{array}{c} k_0\ge 0  \\ k_1,\dots, k_m \ge 2 \\ \sum k_i = n
\end{array}} \left[\begin{array}{c} n \\k_0,\dots,k_m\end{array}\right]_q\,\,
r^{k_0}
\prod_{i=1}^m tq[k_i-1]_{tq}.$$
\end{cor}

Gessel and Reutenauer  \cite{gr} and  Wachs \cite{wa1} derive a major index q-analog of the classical formula for the number of derangements in $\S_n$ (or more generally the number of permutations with a given number of fixed points).  As an immediate consequence of \cite[Corollary 3]{wa1}, one can obtain a $(\maj,\exc)$ generalization.  Since this generalization also follows from Corollary~\ref{expgenthfix} and is not explicitly stated in \cite{wa1},  we state and prove it here.   For all $n \in \N$, let $\mathcal D_n$ be the set of derangements in $\S_n$.

\begin{cor}[of Corollary~\ref{expgenthfix}] \label{excderang} For all $0 \le k \le n$, we have
\begin{equation}\label{fixform} \sum_{\scriptsize \begin{array}{c} \s \in \mathfrak S_n \\ \fix(\s) = k \end{array}} q^{\maj(\s)} t^{ \exc(\s)} =  \left[\begin{array}{c} n \\k\end{array}\right]_q \,\,\sum_{\s \in \mathcal D_{n-k}} q^{\maj(\s)} t^{ \exc(\s)} .\end{equation}
Consequently, \begin{equation}\label{derform} \sum_{\s \in \mathcal D_n} q^{\maj(\s)} t^{ \exc(\s)} = \sum_{k = 0}^n (-1)^{k} q^{k\choose 2} \left[\begin{array}{c} n \\k\end{array}\right]_q A_{n-k}^{\maj,\exc}(q,t).\end{equation}
\end{cor}

\begin{proof} 
Since the left hand side of (\ref{fixform}) equals the coefficient of $r^k\frac{z^n}{[n]_q!}$ on the left hand side of (\ref{expgeneqfix}), we have that the left hand side of (\ref{fixform}) is $\left[\begin{array}{c} n \\k\end{array}\right]_q$ times the coefficient of 
$\frac{z^{n-k}}{[n-k]_q!}$ in 
$(1-tq)/(\exp_q(ztq)-tq\exp_q(z))$.   This coefficient is precisely   $ \sum_{\s \in \mathcal D_{n-k}} q^{\maj(\s)} t^{ \exc(\s)} $.

By summing (\ref{fixform}) over all $k$ and applying Gaussian inversion we obtain (\ref{derform}).
\end{proof}

We point out that our results pertaining to major index have comajor index versions.
We define {\em comajor index}  of  $\sigma \in \S_n$  to be
$$\comaj(\s) := \sum_{i \in [n-1] \setminus \Des(\s)} i = {n \choose 2 } - \maj(\s).$$ (Note that this is different from another commonly used notion of  comajor index.)  For example, we have the following comajor index version of Corollary~\ref{expgenthfix}.

\begin{cor}[of Corollary~\ref{expgenthfix}] \label{comajcor}We have
\begin{equation} \label{expgeneqfixco}
\sum_{n \geq 0}A^{\comaj,\exc,\fix}_n(q,t,r)\frac{z^n}{[n]_q!}=\frac{(1-tq^{-1})\Exp_q(rz)}{\Exp_q(ztq^{-1})-(tq^{-1})\Exp_q(z)}
\end{equation}
\end{cor}

\begin{proof} Use the facts that $[n]_{q^{-1}}! = q^{-{n\choose 2}} [n]_q!$ and $\exp_{q^{-1}}(z) = \Exp_q(z)$ to show that equations (\ref{expgeneqfix}) and (\ref{expgeneqfixco}) are equivalent.
\end{proof}

We also have the following comajor index version of Corollary~\ref{excderang}.
\begin{cor} \label{coexcderang} For all $0 \le k \le n$,
 we have
\begin{equation*} \sum_{\scriptsize \begin{array}{c} \s \in \mathfrak S_n \\ \fix(\s) = k \end{array}} q^{\comaj(\s)} t^{ \exc(\s)} = q^{k\choose 2}  \left[\begin{array}{c} n \\k\end{array}\right]_q \,\,\sum_{\s \in \mathcal D_{n-k}} q^{\comaj(\s)} t^{ \exc(\s)} .\end{equation*}
Consequently, $$ \sum_{\s \in \mathcal D_n} q^{\comaj(\s)} t^{ \exc(\s)}= \sum_{k = 0}^n (-1)^{k}  \left[\begin{array}{c} n \\k\end{array}\right]_q A_{n-k}^{\comaj,\exc}(q,t).$$
\end{cor}

\section{Symmetry and unimodality} \label{secQ}

\subsection{Eulerian quasisymmetric functions} \label{secQ1}
It is well known that the Eulerian numbers $a_{n,j}$ form a symmetric and unimodal sequence for each fixed $n$ (see \cite[p. 292]{com}); i.e., $a_{n,j} = a_{n,n-1-j}$ for all $j=0,1,\dots, n-1$ and $$ a_{n,0} \le a_{n,1} \le  \dots \le a_{n, \lfloor {n-1\over 2} \rfloor} = a_{n, \lfloor {n\over 2} \rfloor} \ge \dots \ge a_{n,n-2}\ge a_{n,n-1} .$$  In this subsection we discuss symmetry and unimodality of  the coefficients of $t^j$ in  the Eulerian quasisymmetric functions and the $q$- and $(q,p)$-Eulerian polynomials.

Let  $f(t):= f_r t^r + f_{r+1} t^{r+1} + \dots + f_{s}t^{s}$ be a nonzero polynomial in $t$ whose coefficients come from a partially ordered ring $R$.  We say that  $f(t)$ is {\em  t-symmetric} with center of symmetry ${ s+r \over 2}$ if   $f_{r+i} = f_{s-i}$ for all $i =0, \dots, s-r$.  We say that  $f(t)$ is also {\em t-unimodal} if
\begin{equation}\label{unimodeq}   f_r \le_R f_{r+1} \le_R \dots \le_R f_{\lfloor {s+r\over 2} \rfloor}= f_{\lfloor {s+r+1\over 2} \rfloor}\ge_R \dots \ge _R f_{s-1} \ge_R f_s.\end{equation}

Let $\Par$ be the set of all partitions of all nonnegative
integers. By choosing a $\Q$-basis $b=\{b_\lambda:\lambda \in
\Par\}$ for the space of symmetric functions, we obtain the partial
order relation on the ring of symmetric functions given by $f \le_b g $ if $g-f$ is $b$-positive,
where a symmetric function is said to be $b$-positive if it is a
nonnegative  linear combination of elements of the basis
$\{b_\lambda\}$. Here we are concerned with the $h$-basis,
$\{h_\lambda : \lambda \in \Par\}$, where $h_\lambda =
h_{\lambda_1} \cdots h_{\lambda_k}$ for $\lambda = (\lambda_1 \ge
\dots \ge \lambda_k)$, and the Schur basis $\{s_{\lambda}:\lambda
\in \Par\}$.
   Since $h$-positivity implies Schur-positivity, the following result also holds for the Schur basis.

\begin{thm} \label{symunimodth} Using the h-basis  to partially order the ring of symmetric functions, we have for all $n,j,k$, \begin{enumerate}
\item the Eulerian quasisymmetric functions $Q_{n,j,k}$ and $Q_{n,j}$ are h-positive symmetric functions,
 \item  the polynomial $\sum_{j=0}^{n-1}Q_{n,j,k} t^j$ is t-symmetric and t-unimodal with center of symmetry ${n-k \over 2}$,
\item   the polynomial $\sum_{j=0}^{n-1}Q_{n,j} t^j$ is t-symmetric and t-unimodal with center of symmetry ${n-1 \over 2}$.
\end{enumerate}
\end{thm}
\begin{proof}

Part (1) is a corollary of Theorem~\ref{introsymgenth} (see also Corollary~\ref{formQcor}).

Parts (2) and (3) follow from Theorem~\ref{introsymgenth} and results of Stembridge~\cite{stem1} on the symmetric function given on the right hand side of (\ref{introsymgenth1}).
For the sake of completeness, we include a proof of Parts (2) and (3)  based on Stembridge's work.

It is  well-known that the product of two symmetric unimodal polynomials in $\N[t]$ with respective centers of symmetry $c_1$ and $c_2$, is symmetric and unimodal with center of symmetry $c_1+c_2$. This result and the proof given in  \cite[Proposition 1.2]{st2} hold more generally for polynomials over an arbitrary partially ordered ring.    By Corollary~\ref{formQcor} we have $$\sum_{j=0}^{n-1}Q_{n,j,k} t^j= \sum_{m = 0}^{\lfloor {n-k \over 2} \rfloor}\,\, \!\!\!\!\sum_{\scriptsize
\begin{array}{c}   k_1,\dots, k_m \ge 2 \\ \sum k_i = n-k
\end{array}}
\!\!\!\!\!\!\! h_{k}
\prod_{i=1}^m h_{k_i} t [k_i-1]_{t}.$$
 Each term  $h_{k}
\prod_{i=1}^m h_{k_i} t [k_i-1]_{t}$ is t-symmetric and t-unimodal with center of symmetry $\sum_{i \ge 0} {k_i \over 2} ={n-k \over 2}$.  Hence the sum of these terms has the same property.

With a bit more effort we show that Part (3) also follows from Corollary~\ref{formQcor}.
We have $$\sum_{j=0}^{n-1}Q_{n,j} t^j = \sum_{j=0}^{n-1} Q_{n,j,1} t^j + \sum_{j=0}^{n-1} Q_{n,j,0} t^j + \sum_{k \ge 2} \sum_{j=0}^{n-1} Q_{n,j,k} t^j .$$
By Part (2) we need only show that $$X(t) :=  \sum_{j=0}^{n-1} Q_{n,j,0} t^j + \sum_{k \ge 2} \sum_{j=0}^{n-1} Q_{n,j,k} t^j$$ is t-symmetric and t-unimodal with center of symmetry ${n-1 \over 2}$.  For any sequence of positive integers $(k_1,\dots,k_m)$,
let $$G_{k_1,\dots,k_m} := \prod_{i=1}^m h_{k_i} t [k_i-1]_{t}.$$  We have by Corollary~\ref{formQcor},
$$ \sum_{j=0}^{n-1} Q_{n,j,0} t^j  =  \sum_{\scriptsize
\begin{array}{c}  m \ge 0 \\ k_1,\dots, k_m \ge 2 \\ \sum k_i = n
\end{array}} G_{k_1,\dots,k_m} $$ and
$$\sum_{k \ge 2} \sum_{j=0}^{n-1} Q_{n,j,k} t^j  =  \sum_{\scriptsize
\begin{array}{c}  m \ge 0 \\ k_1,\dots, k_m \ge 2 \\ \sum k_i = n
\end{array}}h_{k_1} G_{k_2,\dots,k_m}. $$
Hence $$X(t) = \sum_{\scriptsize
\begin{array}{c}  m \ge 0 \\ k_1,\dots, k_m \ge 2 \\ \sum k_i = n
\end{array}} G_{k_1,\dots,k_m}+ h_{k_1} G_{k_2,\dots,k_m} . $$
We claim that $G_{k_1,\dots,k_m}+ h_{k_1} G_{k_2,\dots,k_m} $ is t-symmetric and t-unimodal with center of symmetry ${n-1 \over 2}$.  Indeed, we have
$$G_{k_1,\dots,k_m}+ h_{k_1} G_{k_2,\dots,k_m} = h_{k_1}(t[k_1 - 1]_t + 1) G_{k_2,\dots,k_m}.$$  Clearly $t[k_1 - 1]_t + 1 = 1+t+\dots t^{k_1-1}$ is t-symmetric and t-unimodal with center of symmetry ${k_1-1 \over 2}$, and $G_{k_2,\dots,k_m}$ is   t-symmetric and t-unimodal with center of symmetry ${n-k_1 \over 2}$.  Therefore our claim holds and $X(t) $ is  t-symmetric and t-unimodal with center of symmetry ${n-1 \over 2}$. \end{proof}

We now obtain  analogous results for $ A_n^{\maj,\des,\exc,\fix}$ and $A_n^{\maj,\des,\exc}$ by applying  the  principal specializations.
For the ring of polynomials $\Q[\bf q]$, where ${\bf q}$ is a list of indeterminates, we use the partial order relation: for  $f({\bf q}),g ({\bf q}) \in\Q[({\bf q})]$, define $f({\bf q}) \le_{{\bf q}} g({\bf q})$ if $g({\bf q})-f({\bf q})$ has nonnegative coefficients.

\begin{lemma}  \label{hposqpos}  If  $f$ is a Schur positive homogeneous symmetric function of degree $n$ then
 $(q;q)_{n} \Lambda( f )$ is a polynomial in $q$ with  nonnegative coefficients and $(p;q)_{n+1}\sum_{m\ge 0}  \Lambda_m ( f )p^m$ is a polynomial in $q$ and $p$ with  nonnegative coefficients.  Consequently if $f$ and $g$ are homogeneous symmetric functions of degree $n$ and  $f \le_{{\rm Schur}} g$ then $$(q;q)_{n} \Lambda (f) \le_{(q)} (q;q)_{n} \Lambda (g) $$ and $$(p;q)_{n+1} \sum_{m\ge 0} \Lambda_m( f  )p^m\le_{(q,p)} (p;q)_{n+1} \sum_{m\ge 0} \Lambda_m (g) p^m .$$
 \end{lemma}

 \begin{proof}  This follows from Lemma~\ref{desspec} and the fact that Schur functions are nonnegative  linear combinations of  fundamental quasisymmetric functions (cf. \cite[pp. 360-361]{st3}).
 \end{proof}

\begin{thm} \label{unieulerth}  For all $n,k$,  let $$A^{\maj,\des,\exc}_{n,k}(q,p,t) = \sum_{\scriptsize \begin{array}{c} \s \in \sg_n \\ \fix (\s) = k\end{array}} q^{\maj(\s)} p^{\des(\s)}t^{\exc(\s)}.$$  Then  \begin{enumerate}
\item   $A^{\maj,\des,\exc}_{n,k}(q,p,q^{-1}t)$ is $t$-symmetric with center of symmetry ${n-k \over 2}$
 \item  $A^{\maj,\des,\exc}_{n,0}(q,p,q^{-1}t)$ is t-symmetric and t-unimodal, with center of symmetry ${n \over 2}$
  \item $A^{\maj,\des,\exc}_{n,k}(q,1,q^{-1}t)$ is t-symmetric and t-unimodal, with center of symmetry ${n-k \over 2}$
\item  $A^{\maj,\des,\exc}_{n}(q,1,q^{-1}t)$ is t-symmetric and t-unimodal, with center of symmetry ${n-1 \over 2}$.
\end{enumerate}
\end{thm}

\begin{proof}
Since h-positivity implies Schur positivity, we can use Lemma~\ref{hposqpos} to specialize Theorem~ \ref{symunimodth}.  By  Lemma~\ref{nonstable}, Parts (1)  and (2) are obtained by specializing Part (2) of Theorem~ \ref{symunimodth}.
By (\ref{stablespec1}), Parts (3) and (4) are obtained by specializing Parts (2) and (3) of Theorem~ \ref{symunimodth},  respectively.
\end{proof}

\begin{remark}  The $p,q=1$ case of  Part (2) of Theorem~\ref{unieulerth} is due to Brenti \cite[Corollary 1]{bre1}.
\end{remark}

\begin{remark} One can obtain a stronger result than  Theorem~\ref{unieulerth} by using an unpublished result of Gessel (Theorem~\ref{gesth} below) along with Theorem~\ref{introsymgenth}.  It follows from  (\ref{ges})  and Theorem~\ref{introsymgenth} that 
$$\sum_{j=0}^{n-1} Q_{n,j} t^j = \sum_{d= 0}^{\lfloor \frac {n-1} {2} \rfloor } f_{n,d} \,  t^d (1+t)^{n-1-2d},$$ and it follows from 
(\ref{ges2}) and Theorem~\ref{introsymgenth} that $$\sum_{j=0}^{n-1} Q_{n,j,k} t^j = \sum_{d= 0}^{\lfloor \frac {n-k} {2} \rfloor } f_{n,k,d} \,  t^d (1+t)^{n-k-2d},$$
where $f_{n,d}$ and $f_{n,k,d}$ are Schur positive symmetric functions of degree $n$ (see \cite{sw3} for further details).  Now by Lemma~\ref{hposqpos} we have the following  strengthening of Theorem~\ref{unieulerth}: \begin{enumerate} 
 \item $A_{n,0}^{\maj,\des, \exc}(q,p,q^{-1}t) $ has coefficients in  $\N[q,p]$ when expanded in the basis $\{t^d(1+t)^{n-2d}\}_{d=0}^{\lfloor n /2 \rfloor }$
\item $A_{n,k}^{\maj,\des, \exc}(q,1,q^{-1}t) $ has coefficients in  $\N[q]$ when expanded in the basis $\{t^d(1+t)^{n-k-2d}\}_{d=0}^{\lfloor (n-k)/2 \rfloor }$
\item  $A_{n}^{\maj, \exc}(q,1,q^{-1}t) $ has coefficients in  $\N[q]$ when expanded in the basis $\{t^d(1+t)^{n-1-2d}\}_{d=0}^{\lfloor (n-1) /2 \rfloor }$.
\end{enumerate}  See \cite{bra} for a discussion of the third  result in the special case that $q=1$, that is for the Eulerian polynomials $A_n(t)$.  \end{remark}

   In
Section~\ref{symcyclesec} we conjecture\footnote{
This conjecture was recently proved jointly with A. Henderson.  See Section~\ref{newdevsec} (New developments).}  that the $t$-symmetric polynomial 
$A^{\maj,\des,\exc}_{n,k}(q,p,q^{-1}t)$  is t-unimodal even when $k \ne 0$ and $p \ne 1$.     However, $A^{\maj,\des,\exc}_{n}(q,p,q^{-1}t)$ is not $t$-symmetric as can be seen in the  computation,

\beq A^{\maj,\des,\exc}_4(q,p,q^{-1}t) = 1\! \!\!&+&\!\! \! (3 p + 2pq+pq^2+2p^2q^2+ 2p^2q^3+p^2q^4) t \\ &+&\!\!\! (3p +pq +p^2q +3 p^2q^2 +2 p^2 q^3 +p^3q^4)  t^2\\  &+&\!\!\! p t^3.
\eeq

The Eulerian numbers possess a  property stronger than unimodality, namely log-concavity (see \cite[p. 292]{com}).  A sequence $(s_0,\dots,s_n)$ of real numbers  or polynomials with real coefficients  is said to be {\em log-concave} if 
$s_i^2 -s_{i-1}s_{i+1} \ge 0$ for all $i=1,\dots,n-1$.  Clearly if $(s_0,\dots,s_n)$ is log-concave and each $s_i > 0$ then $(s_0,\dots,s_n)$ is unimodal.  It is natural to ask whether the unimodality results of this subsection have log-concavity versions.     Using the Maple package ACE \cite{ace}, we obtain
$$Q_{5,1}^2 - Q_{5,2} Q_{5,0} = h_{5,4,1}-h_{5,2,2,1}++h_{5,3,2}+4h_{4,3,2,1}+h_{4,4,1,1}+4h_{3,3,2,2}.$$
Hence $Q_{5,1}^2 - Q_{5,2} Q_{5,0}$ is not h-positive.  However, using ACE we have verified the following conjecture for $n\le 8$.
\begin{con} For all $n,j,k$, the symmetric functions $Q_{n,j}^2 - Q_{n,j+1} Q_{n,j-1}$ and  $Q_{n,j,k}^2 - Q_{n,j+1,k} Q_{n,j-1,k}$ are Schur-positive.
\end{con} 

\begin{con} For all $n,k$, the coefficients of  the $t$-polynomials $A_{n,k}^{\maj,\des,\exc}(q,p,q^{-1}t)$ and  $A_{n}^{\maj,\des,\exc}(q,1,q^{-1}t)$ form  log-concave sequences of polynomials in $q$ and $p$.
\end{con}

\subsection{Cycle type Eulerian quasisymmetric functions}  \label{symcyclesec} By means of ornaments we extend the symmetry  results of Section~\ref{secQ1} to the refinements  $Q_{\lambda,j}$ and $A^{\maj,\des,\exc}_{\lambda}(q,p,t)$, where $A^{\maj,\des,\exc}_{\lambda}(q,p,t)$ is the
{\em cycle type $(q,p)$-Eulerian polynomial} defined for each partition $\lambda\vdash n$  by
$$A^{\maj,\des,\exc}_{\lambda}(q,p,t) := \sum_{\scriptsize\begin{array}{c} \s \in \sg_n \\ \lambda(\s) = \lambda  \end{array}}      q^{\maj(\s)}  p^{\des(\s)} t^{\exc(\s)} .$$

\begin{thm} \label{symquasith} For all $\lambda \vdash n$ and $j = 0,1,\dots,n-1$, the Eulerian quasisymmetric function $Q_{\lambda,j}$ is a symmetric function.
\end{thm}
\begin{proof}This result can be proved using plethysm;  it  follows immediately from Corollary~\ref{plethform} below by applying the plethystic inverse of $\sum_{n\ge 0} h_n$ to both sides of (\ref{conj2}).  Here we give a bijective proof involving ornaments. 

For each $k \in \PP$, we will construct an involution
$\psi$ on necklaces that  exchanges the number of occurrences of
the letter $k$ (barred or unbarred) and $k+1$ (barred or unbarred)
in a necklace, but preserves the number of occurrences of all
other letters and  the total number of bars. The result will then
follow from Corollary~\ref{ornth}. We start with necklaces that
contain only the letters $\{k,\bar k, k+1,\overline{ k+1}\}$. Let
$R$ be such a necklace.  We may assume without loss of generality
that $k=1$. First replace all $1$'s with $2$'s and all $2$'s with
$1$'s, leaving the bars in their original positions. The problem
is that now each $2$ that is followed by a $1$ lacks a bar (call
such a 2 {\em bad}) and each $1$ that is followed by a $2$ has a
bar (call such a 1 {\em bad}).  Since the number of bad 1's equals
the number of bad 2's, we can move all the bars from the bad 1's
to the bad 2's, thereby obtaining a necklace $R^\prime$ with the
same number of bars as $R$ but with the number of $1$'s and $2$'s
exchanged.  Let $\psi(R) = R^\prime$.  Clearly $\psi^2(R) = R$.
For example if $R= (2\bar 21\bar 1122\bar 2\bar 211111)$ then we
get $(1\bar 12\bar 2211\bar 1\bar 122222)$ before the bars are
adjusted.  After the bars are moved we have $\psi(R)= (1 12\bar
2\bar 211\bar 112222\bar 2)$.

Now we handle the  case in which $R$ has  (barred and unbarred) $k$'s and $k+1$'s, and other letters which we will call intruders.   The intruders enable us to form  linear segments of $R$ consisting only of (barred and unbarred) $k$'s and $k+1$'s.  To obtain such a linear segment start with a letter of value $k$ or $k+1$ that follows an intruder and read the letters of $R$ in a clockwise direction until another intruder is encountered.   For example if
\bq \label{symnec} R= (\bar 5 334\bar 4\bar 33 \bar 33 6\bar 6 \bar 33 3\bar 4 2 4 4)\eq and $k=3$ then the segments are $334\bar 4\bar 33 \bar 33$,  $\bar 333 \bar 4$, and  $ 4  4$.

 There are two types of segments, even segments and odd segments.  An even (odd) segment contains an even (odd) number of switches, where a switch is a letter of value $k$ followed by one of value $k+1$ or a letter of value $k+1$ followed by one of value $k$. We handle the even and odd segments differently.   In an even segment, we replace all  $k$'s with $k+1$'s and all $k+1$'s with $k$'s, leaving the bars in their original positions.  Again, at the switches, we have  bad $k$'s and bad $k+1$'s,  which are equinumerous.  So we move all the bars from the bad $k$'s to the bad $k+1$'s to obtain a good segment.   This preserves the number of bars and exchanges the number of $k$'s and $k+1$'s.  For example, the even segment $334\bar 4\bar 33 \bar 33$ gets replaced by $4 43\bar 3\bar 44 \bar 44$ before the bars are adjusted.  After the bars are moved we have
 $4 \bar 433 \bar 44 \bar 44$.

An odd segment either starts with a $k$ and ends with a $k+1$ or
vice-verse. Both cases are handled similarly.  So suppose we have
an odd segment of the form $$k^{m_1}(k+1)^{m_2}k^{m_3}(k+1)^{m_4}
\dots k^{m_{2r-1}} (k+1)^{m_{2r}},$$ where each $m_i >0$ and the
bars have been suppressed.   The number of switches is $2r-1$.  We
replace it with the odd segment
$$k^{m_2}(k+1)^{m_1}k^{m_4}(k+1)^{m_3} \dots k^{m_{2r}}(k+1)^{m_{2r-1}},$$
and put  bars  in their original positions.  Again  we may have
created bad $k$'s (but not bad $k+1$'s);  so we need to move some
bars around.  The  positions of the bad $k$'s are in the set
$\{N_1+m_2, N_2+m_4, \dots,N_{r}+m_{2r}\}$, where
$N_i=\sum_{t=1}^{2i-2} m_t$. If there is a bar on the $k$ in
position $N_i+m_{2i}$ we move it to position $N_i+m_{2i-1}$, which
had been barless. This preserves the number of bars and exchanges
the number of $k$'s and $k+1$'s. For example, the odd segment
$\bar 333 \bar 4$ gets replaced by $\bar 344 \bar 4$ before the
bars are adjusted.  After the bars are moved we have $ 34\bar 4
\bar 4$.

Let $\psi(R)$ be the necklace obtained by replacing all the segments in the way described above.
For example if $R$ is the necklace given in (\ref{symnec}) then
$$\psi(R) = (\bar 5 4 \bar 433 \bar 44 \bar 44 6\bar 6 34\bar 4 \bar 42 33).$$

It is easy to check that $\psi^2(R) = R$ for all necklaces $R$.
 Now extend the involution $\psi$ to  ornaments by applying $\psi$ to each necklace of the ornament.

\end{proof}

\begin{thm} \label{symth2}For all $\lambda \vdash n$ with exactly $k$ parts equal to 1, and $j = 0,\dots, n-k$,
  $$Q_{\lambda,j} = Q_{\lambda, n-k-j}.$$
  \end{thm}

 \begin{proof}
 We construct a type-preserving involution $\gamma$ on ornaments.  Let $R$ be an ornament of type $\lambda$.  To obtain $\gamma(R)$, first we bar each unbarred letter of each nonsingleton necklace of $R$ and unbar each barred letter.  Next for each i, we replace each occurrence of  the $i$th smallest value  in $R$ with the $i$th largest value leaving the bars intact.  Clearly $\gamma(R)$ is an ornament with $n-k-j$ bars whenever $R$ is an ornament with $j$ bars.  Also $\gamma^2(R) = R$.    The result now follows from the fact that  $Q_{\lambda,j} $ is symmetric (Theorem~\ref{symquasith}) and from Corollary~\ref{ornth}.
   \end{proof}

   \begin{remark} Although the equation
   \begin{equation} \label{symth3} Q_{n,j} = Q_{n, n-1-j}\end{equation} follows easily from Theorem~\ref{introsymgenth}, we can give a bijective proof by means of banners, using Theorem~\ref{banprop}.  We construct
 an involution $\tau$ on $\bigcup_{\lambda \vdash n}  \mathfrak B_{\lambda,j}$ that is similar to the involution $\gamma$ used in the proof of Theorem~\ref{symth2}.  Let $B$ be a banner of length $n$ with $j$ bars.  To obtain $\tau(B)$,  first we bar each unbarred letter of $B$, except for the last letter,  and unbar each barred letter.  Next for each i, we replace each occurrence of  the $i$th smallest value  in $B$ with the $i$th largest value, leaving the bars intact.  Clearly $\tau(B)$ is a banner of length $n$, with $n-1-j$ bars.
    \end{remark}

Since  $Q_{n,j,k}$ and $Q_{n,j}$ are h-positive, one might expect that the same is true for the refinement $Q_{\lambda,j}$.  It turns out that this is  not the case as the following  computation using the Maple package ACE \cite{ace} shows,
\begin{equation} \label{counterex} Q_{(6),3} = 2 h_{(4,2)} - h_{(4,1,1)} + h_{(3,2,1)} + h_{(5,1)}.\end{equation}
Therefore Theorem~\ref{symunimodth}   cannot hold  for $Q_{\lambda,n}$, as stated.
However, by expanding in the Schur basis we obtain
$$Q_{(6),3} = 3 s_{(6)} + 3 s_{(5,1)} +  3 s_{(4,2)} + s_{(3,3)} + s_{(3,2,1)},$$  which establishes Schur positivity of $Q_{(6),3}$.  We have  used    ACE  \cite{ace}  to confirm the following conjecture for $\lambda \vdash n $ up to $n=8$.

\begin{con}\footnote{
This conjecture was recently proved jointly with A. Henderson.  See Section~\ref{newdevsec} (New developments).}  \label{sposconj} Let $\lambda \vdash n$.  For all  $j=0,1,\dots,n-1$, the Eulerian quasisymmetric function
$Q_{\lambda,j}$ is  Schur positive.  Moreover, $Q_{\lambda,j} - Q_{\lambda,j-1}$ is Schur positive if $1\le j \le \frac{n-k}2$, where $k$ is the number of parts of $\lambda$ that are equal to $1$.  Equivalently, the t-symmetric polynomial $\sum_{j=0}^{n-1}Q_{\lambda,j}t^j$ is   t-unimodal under the  partial order relation on  the ring of symmetric functions induced by the Schur basis.  \end{con}

\begin{remark} We have also verified using ACE that $Q_{\lambda,j}^2 - Q_{\lambda,j+1} Q_{\lambda,j-1}$ is Schur-positive for all  $\lambda \vdash n$ where $n \le 8$. Thus it is reasonable to conjecture that this is true for all $n,j$, and  that  log-concavity versions of the  unimodality conjectures given below, hold. \end{remark}

From Stembridge's work \cite{stem2} one obtains a nice combinatorial description of the
coefficients in the Schur function expansion of $Q_{n,j,k}$.   It would be interesting to do the same for the refinement
$Q_{\lambda,j}$,  at least when $\lambda= (n)$.  In Section~\ref{repthsec} we discuss the  expansion in the power sum symmetric function basis.

By Lemmas~\ref{nonstable} and~\ref{hposqpos}, we have the following consequence of Theorem~\ref{symth2} (and refinement of Part (1) of Theorem~\ref{unieulerth}).

\begin{thm} \label{desmajsym}    Let $\lambda$ be a partition of $n$ with exactly $k$ parts of size $1$.   Then $A^{\maj,\des,\exc}_{\lambda}(q,p,q^{-1}t)$ is $t$-symmetric with center of symmetry ${n-k\over 2}$.
    \end{thm}

The $q,p=1$ case of the following conjecture was proved by Brenti \cite[Theorem 3.2]{bre2}.

\begin{con}\footnote{This conjecture was recently proved jointly with A. Henderson.  See Section~\ref{newdevsec} (New developments).}  \label{sposconj2} Let $\lambda$ be a partition of $n$ with exactly $k$ parts of size 1.  Then the
t-symmetric polynomial $A_{\lambda}^{\maj,\des,\exc}(q,p,q^{-1}t)$ is t-unimodal (with center of symmetry ${n-k \over 2}$).  Consequently for all $n,k$,  $A_{n,k}^{\maj,\des,\exc}(q,p,q^{-1}t)$ is $t$-symmetric and $t$-unimodal with center of symmetry ${n-k \over 2}$.
\end{con}

The next result is easy to prove  and does not rely on the $Q_{\lambda,j}$.
\begin{prop} \label{easyprop} Let $\lambda$ be a partition of $n$ with exactly $k$ parts of size $1$.   Then for all $j=0,1,\dots,n-1$, $$a_{\lambda,j}(q,p) = a_{\lambda,n-k-j}(1/q,q^np).$$
\end{prop}

\begin{proof}  This follows from the type preserving involution $\phi: \sg_n \to \sg_n$ defined by letting
$\phi(\s) $ be obtained from $\s$ by writing $\s$ in cycle form and replacing each
$i$ by $n-i+1$.   Clearly $\phi$  sends permutations with $j$ excedences to permutations with $n-k-j$ excedences.   Also $i \in \Des(\s)$ if and only if $n-i \in \Des(\phi(\s))$.
It follows that $\phi$ preserves $\des$ and  $$\maj(\phi(\s)) = \des(\s) n- \maj (\s).$$
\end{proof}

By combining Proposition~\ref{easyprop} and Theorem~\ref{desmajsym}, we obtain the following result.
\begin{cor} Let $\lambda$ be a partition of $n$ with exactly $k$ parts of size $1$.  Then for all $j=0,1,\dots,n-1$,
$$a_{\lambda,j}(q,p) = q^{2j+k-n} a_{\lambda,j}(1/q,q^np).$$ Equivalently,  the coefficient of $p^dt^j$ in $A_{\lambda}^{\maj,\des,\exc}(q,p,t)$ is $q$-symmetric with center of symmetry $j + {k+nd -n \over 2}$.\end{cor}

\begin{con} The coefficient of $p^dt^j$ in $A_{\lambda}^{\maj,\des,\exc}(q,p,t)$ is $q$-unimodal.
\end{con}

\section{Further properties 
} \label{repthsec}

 The ornament characterization of $Q_{\lambda,j}$ yields  a plethystic formula expressing $Q_{\lambda,j}$ in terms of $Q_{(i),k}$.   (Note $(i)$ stands for the partition with a single part $i$.) Let $f$ be a symmetric function in variables $x_1,x_2,\dots$ and let $g$ be a formal power series with positive integer coefficients.  Choose any ordering of the monomials of $g$, where a monomial appears in the ordering $m$ times if its coefficient is $m$.  The {\em plethysm} of $f$ and $g$, denoted $f[g]$ is defined to be the the formal power series obtained from $f$ by replacing $x_i$ by the $i$th monomial of $g$, for each $i$.  The following result is an immediate consequence of  Corollary~\ref{ornth}.

\begin{cor} \label{plethform} Let $\lambda$ be a partition with $m_i$ parts of size $i$ for all  $i$.
Then \begin{equation}\label{conj}\sum_{j=0}^{|\lambda|-1} Q_{\lambda,j } t^j = \prod_{i\ge 1} h_{m_i}[ \sum_{j=0}^{i-1} Q_{(i), j} t^j].\end{equation}
Consequently,
\begin{equation}\label{conj2}\sum_{n,j \ge 0} Q_{n,j } t^j z^n  = \sum_{n\ge 0} h_n[\sum_{i,j\ge 0}Q_{(i), j} t^j z^i ] .\end{equation}
\end{cor}

By specializing (\ref{conj}) we obtain the following result. Recall that $(\lambda,\mu)$ denotes the partition of $m+n$ obtained by  concatenating  $\lambda$ and $\mu$ and then appropriately rearranging the parts.

\begin{cor} Let $\lambda$ and $\mu$ be  partitions of $m$ and $n$, respectively.  If  $\lambda$ and $\mu$ have no  common parts then
$$A^{\maj,\exc}_{(\lambda,\mu)}(q,t) = \left[ \begin{array}{c} m+n \\ m \end{array}\right]_q \,\, A^{\maj,\exc}_{\lambda}(q,t) A^{\maj,\exc}_{\mu}(q,t).$$
Consequently, if $\lambda$ has no parts equal to $1$ then for all $j$,
$$a_{(\lambda, 1^m),j}(q,1) = \left[ \begin{array}{c} m+n \\ m \end{array}\right]_q \,\, a_{\lambda,j}(q,1).$$
\end{cor}

\begin{cor} \label{pleth} For all $j$ and $k$ we have $$Q_{(2^j,1^k),j} = h_j[h_2] h_k.$$
 Moreover,
 \begin{equation}\label{invol} \sum_{\lambda,j} Q_{\lambda,j} t^j z^{|\lambda|} = \prod_{i\ge 1} (1-x_iz)^{-1} \prod_{1 \le i \le j} (1-x_ix_j t z^2)^{-1},\end{equation}
where $\lambda$ ranges over all partitions with no parts greater than $2$.
\end{cor}

By specializing (\ref{invol}) one can obtain formulas for the  generating functions of the $(\maj,\des,\exc)$ and the $(\maj, \exc)$ enumerators of involutions.  Since $\exc$ and $\fix$ determine each other for involutions, these formulas can be immediately obtained from the formulas that Gessel and Reutenauer derived for $\maj,\des,\fix$ in \cite[Theorem 7.1]{gr}.

It follows from Corollary~\ref{plethform} that in order to prove the conjecture on Schur positivity of    $Q_{\lambda,j}$ (Conjecture~\ref{sposconj}), it suffices to prove it  for all $\lambda$ with a single part because the plethysm of Schur positive symmetric functions is a Schur positive symmetric function.   Note that  if $\lambda$ has no parts greater than $2$, then by Corollary~\ref{pleth}, we conclude that  $Q_{\lambda,j}$ is Schur positive.

The Frobenius characteristic  is a fundamental isomorphism
from the ring of  virtual representations of  the symmetric groups to the ring
of  symmetric functions over the integers.  We recall the definition here.
For  a virtual representation $V$ of $\sg_n$,  let $\chi^V_\lambda$ denote the value of the  character of $V$ on the conjugacy  class of type $\lambda$.  If  $\lambda$ has $m_i$ parts of size $i$ for each $i$, define
$$ z_\lambda := 1^{m_1} m_1! 2^{m_2} m_2! \cdots n^{m_n} m_n!.$$
Let $$p_\lambda := p_{\lambda_1} \cdots p_{\lambda_k},$$ for $\lambda = (\lambda_1,\dots, \lambda_k), $ where $p_n$ is the power sum symmetric function $\sum_{ i \ge 1} x_i^n$.
The Frobenius characteristic of a virtual representation $V$ of $\sg_n$ is defined as follows:
$$\ch V:=\sum_{\lambda \vdash n}  z_\lambda^{-1}\,\, \chi^V_\lambda\,\, p_\lambda .$$

Recall that the Frobenius characteristic of a virtual representation
$\rho$ of $\sg_n$ is $h$-positive if and only if $\rho$ arises from a
permutation representation in which every point stabilizer is a Young subgroup. 
Hence, by part (1) of  Theorem~\ref{symunimodth}, $Q_{n,j,k}$ is the Frobenius
characteristic of a representation of $\sg_n$ arising from such a
permutation representation.
However (\ref{counterex})
shows that this is not the case in general for the refined
Eulerian quasisymmetric functions $Q_{\lambda,j}$. (It can be
shown that $Q_{(6),3}$ is not the Frobenius characteristic of any
permutation representation at all.) Recall also that the Frobenius
characteristic of a virtual representation is Schur positive if
and only if it is an actual representation. Hence if
$V_{\lambda,j}$ is the virtual representation  whose Frobenius
characteristic is $Q_{\lambda,j}$ then Conjecture~\ref{sposconj}
says that $V_{\lambda,j}$ is  an actual representation.

\begin{prop} \label{dimvprop} Let $\lambda \vdash n$.  Then the dimension of $V_{\lambda,j}$ equals the number of permutations of cycle type $\lambda$ with $j$ excedances.  Moreover, the dimension of $V_{(n),j}$ is the Eulerian number $a_{n-1,j-1}$.
\end{prop}

\begin{proof} The dimension of $V_{\lambda,j}$ is the coefficient of $x_1x_2 \cdots x_n$ in $Q_{\lambda,j}$.  Using the definition of $Q_{\lambda,j}$, this is the number of permutations of
cycle type $\lambda$ with $j$ excedances.

The set of  $n$-cycles with $j$ excedances maps bijectively to the set $\{\s \in \sg_{n-1} : \des(\s) = j-1\}$.  Indeed, the bijection is obtained by writing the cycle in the form $(c_1,\dots,c_{n-1},n)$ and then extracting the word $c_{n-1}\cdots c_1$.  Hence the number of $n$-cycles with $j$ excedances is the Eulerian number $a_{n-1,j-1}$.
\end{proof}

 We have a conjecture\footnote{This conjecture has now been proved by Sagan, Shareshian and Wachs \cite{ssw}} for the character of $V_{(n),j}$, which generalizes Proposition~\ref{dimvprop}.  We have confirmed our conjecture  up to $n=8$ using the Maple package ACE \cite{ace}.  Equivalently, our conjecture gives the coefficients in the expansion of  $Q_{(n),j} $ in the basis of power sum symmetric functions and implies that  $Q_{(n),j} $ is $p$-positive.  For a polynomial $F(t)=a_0 + a_1 t + ... + a_k t^k$ and a positive integer
$m$, define $F(t)_m$ to be the polynomial obtained from $F(t)$ by erasing all terms
$a_i t^i$ such that $\gcd(m,i) \neq 1$.  For example, if $F(t)=1+t+2t^2+3t^3$
then $F(t)_2=t+3t^3$.
For a partition $\lambda= (\lambda_1,\lambda_2, \dots, \lambda_k)$, define $$g(\lambda): =\gcd(\lambda_1,...,\lambda_k).$$

\begin{con}\hspace{-.1in}$^{1}$ For $\lambda = (\lambda_1,\dots,\lambda_k)\vdash n$,
let  $$G_\lambda(t) := \left ( t A_{k-1}(t) \prod_{i=1}^{k}
[\lambda_i]_t \right)_{g(\lambda)}.$$  Then $$\sum_{j=0}^{n-1}
Q_{(n),j} t^j =\sum_{\lambda \vdash n}  z_\lambda^{-1}
G_\lambda(t) p_\lambda. $$ Equivalently, the character of
$V_{(n),j}$ evaluated on conjugacy class $\lambda$ is the
coefficient of $t^j$ in $G_{\lambda}(t)$. \label{cvalcon}
\end{con}

Conjecture \ref{cvalcon} holds whenever $\lambda_k=1$, see
Corollary \ref{cvalcor}. It follows from this  that the
character of $V_{(n),j}$ evaluated at $\lambda = 1^n$ is the
Eulerian number $a_{n-1,j-1}$. This special case of the conjecture
is also a consequence of  Proposition~\ref{dimvprop} because any character
evaluated at $1^n$ is the dimension of the representation. In the
tables below we give the character values for $V_{(n),j}$ at
conjugacy class $\lambda$, which were computed using  ACE
\cite{ace} and confirm the conjecture up to $n=8$.  The columns
are indexed by $n,j$ and the rows by the partitions $\lambda$.

{\small 
$$
\begin{array}{|c||c|c|}
\hline & 4,1 & 4,2   \\ \hline\hline 4 & 1 & 0
\\ \hline 31 & 1& 1 \\ \hline 2^2 & 1 & 0
\\ \hline 21^2 & 1 & 2
\\ \hline 1^4 & 1 & 4 \\ \hline
\end{array}\hspace{.3in} \begin{array}{|c||c|c|}
\hline & 5,1 & 5,2   \\ \hline\hline 5 & 1 & 1
\\ \hline 41 & 1& 1 \\ \hline 32 & 1 & 2
\\ \hline 31^2 & 1 & 2
\\ \hline 2^2 1 & 1 & 3
\\ \hline 2 1^3 & 1 & 5
\\ \hline 1^5 & 1 & 11 \\ \hline
\end{array}\hspace{.3in}
\begin{array}{|c||c|c|c|}
\hline & 6,1 & 6,2  & 6,3 \\ \hline\hline 6 & 1 & 0 & 0
\\ \hline 51 & 1& 1 & 1
\\ \hline 42 & 1 & 0 & 2
\\ \hline 3^2 & 1 & 2 & 0
\\ \hline 41^2 & 1 & 2 & 2
\\ \hline 321 & 1 & 3 & 4
\\ \hline 2^3 & 1 & 0 & 6
\\ \hline 3 1^3& 1 & 5 & 6
\\ \hline 2^21^2 & 1 & 6 & 10
\\ \hline 2  1^4  &1 & 12 &  22
\\ \hline 1^6& 1 & 26 & 66
\\ \hline
\end{array}$$}
\vfil
\newpage

{\small \vspace{.1in}$$\begin{array}{|c||c|c|c|}
\hline & 7,1 & 7,2  & 7,3
\\ \hline\hline 7 & 1 & 1 & 1
\\ \hline 61 & 1& 1 & 1
\\ \hline 52 & 1 & 2 & 2
\\ \hline 43 & 1 & 2 & 3
\\ \hline 51^2 & 1 & 2 & 2
\\ \hline 421 & 1 & 3 & 4
\\ \hline 3^21 & 1 & 3 & 5
\\ \hline 32^2 & 1 & 4 & 7
\\ \hline 4 1^3& 1 & 5 & 6
\\ \hline 321^2 & 1 & 6 & 11
\\ \hline 2^31 & 1 & 7 & 16
\\ \hline 3  1^4  &1 & 12 &  23
\\ \hline 2^21^3 & 1 & 13 & 34
\\ \hline 21^5 & 1 & 27 & 92
\\ \hline 1^7& 1 & 57 & 302
\\ \hline
\end{array} \hspace{.3in}
\begin{array}{|c||c|c|c|c|}
\hline & 8,1 & 8,2  & 8,3 & 8,4
\\ \hline\hline 8& 1 & 0 & 1 & 0
\\ \hline 71 & 1& 1 & 1 &1
\\ \hline 62 & 1 & 0 & 2 & 0
\\ \hline 53 & 1 & 2 & 3 & 3
\\ \hline 4^2 & 1 & 0 & 3 & 0
\\ \hline 61^2 & 1 & 2 & 2 & 2
\\ \hline 521 & 1 & 3 & 4 & 4
\\ \hline 431 & 1 & 3 & 5 & 6
\\ \hline 42^2 & 1 & 0 & 7 & 0
\\ \hline 3^22 & 1 & 4 & 8 & 10
\\ \hline 5 1^3& 1 & 5 & 6 & 6
\\ \hline 421^2 & 1 & 6 & 11 & 12
\\ \hline 3^2 1^2 & 1 & 6 & 12 & 16
\\ \hline 32^2 1 & 1 & 7 & 17 & 22
\\ \hline 2^4  &1 & 0 &  23 & 0
\\ \hline 41^4  &1 & 12 &  23 & 24
\\ \hline 321^3 & 1 & 13 & 35 & 46
\\ \hline 2^31^2 & 1 & 14 & 47 & 68
\\ \hline 31^5 & 1 & 27 & 93 & 118
\\ \hline 2^21^4 & 1 & 28 & 119 & 184
\\ \hline 21^6 & 1 & 58 & 359 & 604
\\ \hline 1^8& 1 & 120 &1191& 2416
\\ \hline
\end{array}
$$}

Conjecture \ref{cvalcon} resembles the following immediate consequence of Theorem~\ref{introsymgenth} and Stembridge's   computation  \cite[Proposition 3.3]{stem1} of the coefficients  in the expansion of the $r=1$ case of the right hand side of (\ref{introsymgenth1}), in the basis of power sum symmetric functions.

\begin{prop}\label{stem}  We have 
$$\sum_{j=0}^{n-1} Q_{n,j} t^j = \sum_{\lambda \vdash n} z_{\lambda}^{-1}
\left(  A_{\ell(\lambda)}(t) \prod_{i=1}^{\ell(\lambda)} [\lambda_i]_t\right) p_\lambda,$$
where $\ell(\lambda)$ denotes the length of the partition $\lambda$ and $\lambda_i$ denotes its $i$th part. 
\end{prop}

By using banners we are able to prove the following curious fact
about the representation $V_{(n),j}$.

\begin{thm} \label{resthm}  For all
$j=0,\dots, n-1$,  the restriction of  $V_{(n),j}$ to $\sg_{n-1}$
is  the permutation representation whose Frobenius characteristic
is $Q_{n-1,j-1}$.
\end{thm}

\begin{proof}
Given a homogeneous symmetric function $f$ of degree $n$, let $\tilde f$ be the polynomial obtained from $f$ by setting  $x_i =0$  for all $i > n$.   Since $f$ is symmetric, $\tilde f$ determines $f$.

We use the fact that
for any virtual representation $V$ of $\sg_n$, \begin{equation} \label{vcheq} \widetilde {\ch \,V \!\!\downarrow}_{\sg_{n-1} }= {\partial \over \partial x_n} \widetilde{ \ch V }\,|_{x_n=0}.\end{equation}

By Theorem~\ref{banprop},  $\widetilde{Q_{(n),j}}$ is the sum of weights of Lyndon banners of length $n$, with $j$ bars, whose letters have value  at most $n$.  The partial derivative with respect to $x_n$ of  the weight of a such a Lyndon banner $B$ is $0$  unless  $n$ or $\bar n$ appears in $B$.  Since $B$  is Lyndon, $\bar n$ must be its first letter.  If the partial derivative is not $0$ after setting $x_n=0$ then all the other letters of $B$  must be less in absolute value than $n$.  In this case,  the partial derivative is the weight of the banner $B^\prime$ obtained from $B$ by removing its first letter $\bar n$.
We thus have
\bq \label{partialeq}  {\partial \over \partial x_n} \widetilde{ \ch V_{(n),j}} |_{x_n=0} = \sum_{B^\prime} \wt(B^\prime),\eq
 where $B^\prime$ ranges over the set of all banners obtained by removing the first letter from a Lyndon banner of length $n$ with $j$ bars, whose first letter is $\bar n$ and whose other letters have value strictly less than $n$.     Clearly this is the set of all banners of length $n-1$, with $j-1$ bars, whose letters have value at most $n-1$.
  Thus the sum  on the right hand side of (\ref{partialeq}) is precisely $\widetilde {Q_{n-1,j-1}}$.  It therefore follows from (\ref{vcheq}) that
$$\widetilde {\ch \,V_{(n),j} \!\!\downarrow}_{\sg_{n-1} } = \widetilde {Q_{n-1,j-1}},$$ which implies
$$ {\ch \,V_{(n),j} \!\!\downarrow}_{\sg_{n-1} } =  {Q_{n-1,j-1}},$$
 \end{proof}
 
Theorem \ref{resthm}, along with Proposition~\ref{stem}, allows us  to
prove that Conjecture \ref{cvalcon} holds when $\lambda$ has a part
of size one.

\begin{cor} \label{cvalcor}
Let $\lambda=(\lambda_1,\ldots,\lambda_k=1)$ be a partition of $n$
and let $\sigma \in \S_n$ have cycle type $\lambda$. Then the
character of $V_{(n),j}$ evaluated at $\sigma$ is the coefficient
of $t^j$ in $tA_{k-1}(t)\prod_{i=1}^{k}[\lambda_i]_t$.
\end{cor}

\begin{proof}
We may assume that $\sigma$ fixes $n$, which allows us to think
of $\sigma$ as an element of $\S_{n-1}$.  Let $V_{n-1,j-1}$ be the
representation of $\S_{n-1}$ whose Frobenius characteristic is
$Q_{n-1,j-1}$.  By Theorem \ref{resthm}, the character value in
question is equal to the character value of $\sigma$ on
$V_{n-1,j-1}$.  Corollary
\ref{cvalcor} now follows from Proposition~\ref{stem}.
\end{proof}

\section{Other occurrences} \label{othersec} The Eulerian quasisymmetric functions refine symmetric functions that have appeared earlier in the literature.        A multiset derangement of order $n$ is a $2 \times n$  array  of  positive integers whose top row is weakly increasing, whose bottom row is a rearrangement of the top row, and whose columns contain distinct entries.  An excedance of a multiset derangement $D = (d_{i,j})$  is a column $j$ such that $d_{1,j} < d_{2,j}$.   Given a multiset derangement $D=(d_{i,j})$, let $x^D:=\prod_{i=1}^n x_{d_{1,i}}$.   For all $j < n$, let   $\mathcal D_{n,j}$ be the set of all derangements in $\mathfrak S_n$ with $j$ excedances and let $\mathcal {MD}_{n,j}$ be the set of all multiset derangements of order $n$ with $j$ excedances. Set
$$d_{n,j}({\bf x}) :=  \sum_{D\in \mathcal {MD}_{n,j}} {\rm x}^D.$$ Askey and Ismail \cite{ai} (see also \cite{kz}) proved the following $t$-analog of MacMahon's \cite[Sec. III, Ch. III]{mac1}
result on multiset derangements  \begin{equation}\label{macderang}  \sum_{j,n \ge 0} d_{n,j}({\bf x})t^j z^n= { 1\over 1 - \sum_{i \ge 2} t[i-1]_t e_i z^i},\end{equation}
where $e_i$ is the $i$th elementary symmetric function.
\begin{cor}[to Theorem~\ref{introsymgenth}]  For all $n,j \ge 0 $ we have
\begin{equation} \label{derange1} d_{n,j}({\bf x}) = \omega Q_{n,j,0} = \sum_{\sigma \in \mathcal D_{n,j}} F_{[n-1]\setminus \Exd(\s) ,n},\end{equation} where $\omega$ is the standard involution on the ring of symmetric functions, which takes $h_n$ to $e_n$. Consequently,
\begin{equation}\label{derangeq1}d_{n,j}(1,q,q^2,\dots) = (q;q)_{n}^{-1} \sum_{\sigma \in \mathcal D_{n,j}} q^{\comaj(\s)+j },
\end{equation} and
\begin{equation}\label{derangeq2} \sum_{m\ge 0} p^m \Lambda_m d_{n,j}({\bf x}) = (p;q)_{n+1}^{-1}\sum_{\sigma \in \mathcal D_{n,j}} q^{\comaj(\s)+j} p^{n-\des(\s)+1}.
\end{equation}
\end{cor}

\begin{proof} The right hand side of (\ref{macderang}) can be obtained by applying $\omega$ to the right hand side of  (\ref{introsymgenth2}) and setting  $r=0$.  Hence the first equation of (\ref{derange1}) holds.

The involution  on the ring $\mathcal Q$ of quasisymmetric functions defined  by
$ F_{S,n} \mapsto F_{[n-1]\setminus S,n}$ restricts to  $\omega$ on the ring of symmetric functions (cf. \cite[Exercise 7.94.a]{st3}).  Hence  the second equation of (\ref{derange1})  follows from the definition of $Q_{n,j,0}$.  Equations (\ref{derangeq1}) and (\ref{derangeq2}) now  follow from Lemmas~\ref{desspec} and~\ref{exdlem}.
\end{proof}

Let $W_n$ be the set of all words of length $n$ over alphabet $\PP$ with no adjacent repeats, i.e.,
$$W_n := \{w \in \PP^n : w(i )\ne w(i+1)\,\, \forall i = 1,2,\dots,n-1\}. $$  Define the enumerator
$$Y_n(x_1,x_2,\dots):= \sum_{w\in  W_n} {\bf x}^w,$$ where
${\bf x}^w:=x_{w(1)} \cdots x_{w(n)}$.
In \cite{csv} Carlitz, Scoville and Vaughan prove  the identity
\begin{equation}\label{carl} \sum_{n \ge 0} Y_n({\bf x}) z^n = {\sum_{i \ge 0} e_i z^i \over 1 - \sum_{i \ge 2} (i-1) e_i z^i}.\end{equation}  (See Dollhopf, Goulden and Greene  \cite{dgg} and Stanley \cite{st4} for alternative proofs.)   It was observed by Stanley  that there is a nice generalization of (\ref{carl}).
\begin{thm}[Stanley (personal communication)]  \label{stanth} For all $n,j \ge 0$, define $$Y_{n,j}(x_1,x_2,\dots):= \sum_{\scriptsize \begin{array}{c} w\in  W_n \\ \des(w) = j\end{array}} {\bf x}^w.$$ Then
\begin{equation}\label{carlg} \sum_{n,j \ge 0} Y_{n,j}({\bf x})t^j z^n = {(1-t) E(z) \over E(zt) - t E(z)} ,\end{equation} where  $$E(z) = \sum_{n \ge 0} e_n z^n.$$
\end{thm}

By combining  Theorems~\ref{introsymgenth} and \ref{stanth} we conclude that
$$Q_{n,j} = \omega Y_{n,j}.$$  Stanley (personal communication)  observed that there is a combinatorial interpretation of this identity in terms of P-partition reciprocity \cite[Section 4.5]{st5}.
 Indeed, words in $W_n$ with  fixed descent set $S\subseteq [n-1]$ can be identified with strict $P$-partitions where $P$ is  the poset on $\{p_1,\dots,p_n\}$ generated by cover relations $p_i < p_{i+1}$ if $i \in S$ and $p_{i+1} < p_{i}$ if  $i \notin S$.  Banners of length $n$ in which the set of positions of barred letters equals $S$ can be identified with $P$-partitions for the same poset $P$.  It therefore follows from
P-partition reciprocity that
\begin{equation} \label{recipeq} Y_{n,j}({\bf x}) = \omega \sum_{B} \wt(B), \end{equation}  summed over all banners of length $n$ with $j$ bars.

In \cite{st4} Stanley views words in $W_n$  as proper colorings of a path $P_n$ with $n$ vertices and $Y_n$  as the chromatic symmetric function of $P_n$.  The chromatic symmetric function of a graph $G=(V,E)$ is a symmetric function analog of the chromatic polynomial $\chi_G$ of $G$.   Stanley  \cite[Theorem 4.2]{st4} also defines a symmetric function analog of   $(-1^{|V|})\chi_G(-m)$ which enumerates all pairs $(\eta,c)$ where $\eta$ is an acyclic orientation of $G$ and $c:V\to [m]$ is a coloring satisfying $c(u) \le c(v) $ if $(u,v)$ is an edge of $\eta$.   For $G=P_n$, one can see that these pairs can be identified with banners of length $n$.  Hence Stanley's reciprocity theorem for chromatic symmetric functions \cite[Theorem 4.2]{st4}   reduces to an identity that is refined by (\ref{recipeq}) when $G= P_n$.

Another interesting combinatorial interpretation of the Eulerian quasisymmetric functions comes from Gessel who considers the set $U_n$ of words  of length $n$ over the alphabet $\PP$ with no double (i.e., adjacent)  descents  and no descent in the last position $n-1$, and the set  $\tilde U_n$  of words of length $n$ over the alphabet $\PP$ with no double descents  and no descent in the last position nor the first position. 
\begin{thm}[Gessel (personal communication)] \label{gesth}
\begin{equation} \label{ges} 1+\sum_{n \ge 1} z^n \sum_{w \in U_n} {\bf x}^w t^{\des(w)} (1+t) ^{n-1-2\des(w)} = {(1-t) H(z) \over H(zt) - t H(z)} ,\end{equation}
and \begin{equation}\label{ges2} 1+\sum_{n \ge 2} z^n \sum_{w \in \tilde U_n} {\bf x}^w t^{\des(w)+1} (1+t) ^{n-2-2\des(w)} = {1-t \over H(zt) - t H(z)} ,\end{equation}
where as above $ {\bf x}^w := x_{w(1)}\cdots x_{w(n)}$.
\end{thm}

The symmetric function  on the right hand side of (\ref{ges}) has
also occurred  in the work of   Procesi \cite{pr}, Stanley
\cite{st2}, Stembridge \cite{stem1,stem2}, and Dolgachev and Lunts
\cite{dl}.  They studied a representation of the symmetric group
on the cohomology of the toric variety $X_n$ associated with the
Coxeter complex of $\sg_n$.  (See, for example \cite{Br}, for a
discussion of Coxeter complexes and \cite{fu} for an explanation
of how toric varieties are associated to polytopes.)  The action
of $\sg_n$ on the Coxeter complex determines an action on $X_n$
and thus a linear representation on the cohomology groups of
$X_n$. Now $X_n$ can have nontrivial cohomology only in dimensions $2j$, for $0 \leq j \leq n-1$.  (See for example \cite[Section 4.5]{fu}.) 
The action of $\sg_n$ on $X_n$
induces a representation of $\sg_n$ on the cohomology
$H^{2j}(X_n)$ for each $j = 0,\dots,n-1$. Stanley \cite{st2},
using a formula of Procesi \cite{pr}, proves that
$$\sum_{n\ge 0} \sum_{j=0}^{n-1} \ch H^{2j}(X_n)\,t^{j} z^n
= {(1-t) H(z) \over H(zt) -tH(z)}, $$ 
where $\ch$ denotes the Frobenius characteristic  Combining this with
Theorem~\ref{introsymgenth} yields the following conclusion.

\begin{thm}\label{toricth}  For all $j = 0,1, \dots, n-1$,
$$\ch H^{2j}(X_n)=Q_{n,j}.$$
\end{thm}

Finally, we discuss the occurrence of  the Eulerian quasisymmetric functions and the $q$-Eulerian polynomials that motivated the work in this paper.  This involves the homology of  certain partially ordered sets (posets).   Recall that for a poset $P$, the {\it order complex} $\Delta P$ is the abstract simplicial complex whose vertices are the elements of $P$ and whose $k$-simplices are totally ordered subsets of size $k+1$ from $P$.  The (reduced) homology of $P$ is given by
$\rh_k(P):= \rh_k(\Delta P;\C)$.  (See \cite{w1}  for a discussion of poset homology.) 

  The first poset we consider is the Rees product $(B_n\setminus\{\emptyset\}) * C_n$, where $ B_n$ is the Boolean algebra on $\{1,2,\dots,n\}$ and $C_n$ is an $n$-element chain.  Rees products of posets were introduced by
Bj\"orner and Welker in \cite{bw}, where they study connections between poset topology and commutative algebra.  (Rees products of affine semigroup posets arise from the ring-theoretic Rees construction.)
The Rees product of  two ranked posets  is a subposet of the usual product poset (see  \cite{bw} or \cite{sw2} for the precise definition.).
 
  Bj\"orner and Welker \cite{bw} conjectured and Jonsson \cite{jo} proved  that the dimension of the top homology of $(B_n\setminus\{\emptyset\}) * C_n$ is equal to the number of derangements in $\mathfrak S_n$.  Below we discuss  equivariant versions and $q$-analogs of  both  this result and refinements, which were derived in \cite{sw2} using results of this paper.

   The poset $(B_n\setminus\{\emptyset\}) * C_n$ has $n$ maximal elements  all of rank $n$.  The usual action of $\mathfrak S_n$ on $B_n$ induces  actions of $\mathfrak S_n$ on $(B_n\setminus\{\emptyset\}) * C_n$  and on each lower order ideal generated by a maximal element, which respectively induce  representations of $\mathfrak S_n$ on the homology of  $(B_n\setminus\{\emptyset\}) * C_n$ and on the homology of each open lower order ideal.  
 \begin{thm}[\cite{sw2}] \label{introhomol} Let $x_1,\dots, x_n$ be the maximal elements of $(B_n\setminus\{\emptyset\}) * C_n$.  For each $j=1,\dots, n$, let $I_{n,j}$ be the open lower order ideal generated by $x_j$.  Then $\dim \rh_{n-2}(I_{n,j}) $ equals the Eulerian number $a_{n,j-1}$ and 
 \begin{equation*}
{\rm {ch}}(\rh_{n-2}(I_{n,j}) )=\omega  Q_{n,j-1}.
\end{equation*}
\end{thm}

 In \cite{sw2} we derive  the following equivariant version of the Bj\"orner-Welker-Jonsson result as a cosequence of Theorem~\ref{introhomol},
$$\ch(\tilde H_{n-1}((B_n\setminus \{\emptyset\}) \ast C_n)) =\omega \sum_{j= 0}^{n-1}  Q_{n,j,0}.$$

A $q$-analog of Theorem~\ref{introhomol} is obtained by considering the Rees product   $(B_n(q)\setminus\{0\}) * C_n$, where $B_n(q)$ is the lattice of subspaces of the $n$-dimensional vector space $\F_q^n$ over the $q$ element field $\F_q$.  Like $(B_n\setminus\{\emptyset\}) * C_n$, the q-analog $(B_n(q)\setminus\{0\}) * C_n$ has $n$ maximal elements all of rank $n$.

\begin{thm}[\cite{sw2}] \label{intohomolq} Let $x_1,\dots, x_n$ be the maximal elements of $(B_n(q)\setminus\{0\}) * C_n$.  For each $j=1,\dots, n$, let $I_{n,j}(q)$ be the open lower order ideal generated by $x_j$.
  Then  \begin{equation*} \label{intoq}
  \dim \rh_{n-2}(I_{n,j}(q))=\sum_{\scriptsize\begin{array}{c} \s \in\sg_n \\ \exc(\s) = j-1 \end{array}} q^{\comaj(\s)+j-1},\end{equation*}
  where $\comaj(\s) = {n \choose 2}- \maj(\s)$.
\end{thm}

As a consequence, in \cite{sw2} we obtain the following $q$-analog of the
Bj\"orner-Welker-Jonsson result:\begin{equation*} \dim \tilde
H_{n-1}((B_n(q)\setminus\{0\}) \ast C_n)= \sum_{\sigma \in \mathcal D_n} q^{\comaj(\s) + \exc(\s)},\end{equation*} where $\mathcal
D_n$ is the set of derangements in $\S_n$.
In \cite{sw2}  we also   derive  type BC analogs (in the sense of Coxeter groups) of the Bj\"orner-Welker-Jonsson derangement  result and our $q$-analog.

\section{New developments} \label{newdevsec}
Since a preliminary version of this paper  was first circulated, there have been a number of interesting developments.   Foata and Han \cite{FH2} obtained  a result which generalizes Corollary~\ref{foha} and also reduces to a  type B version of Corollary~\ref{foha}.   Hyatt \cite{hy1}, \cite{hy} extended our notion of bicolored necklaces to multicolored necklaces, enabling him to obtain a generalization of    Theorem~\ref{introsymgenth}, which also reduces to a type B version of  Theorem~\ref{introsymgenth} and also specializes to the Foata-Han result and to another type B analog of Corollary~\ref{foha}.  

Foata and Han \cite{FH3} use the results of this paper  to study a $q$-analog of the secant and tangent numbers. 

Sagan, Shareshian and Wachs \cite{ssw} prove Conjecture~\ref{cvalcon} and use it and other results of this paper to show that the polynomial $A_\lambda^{\maj,\exc}(\omega,\omega^{-1}t)$ has positive integer coefficients if $\omega$ is any $|
\lambda|$th root of unity.  They give these integers  a combinatorial interpretation in the framework of the cyclic sieving phenomenon of Reiner, Stanton and White \cite{rsw}.

Shareshian and Wachs \cite{sw4} introduce and initiate a study of a class of quasisymmetric functions which contain  the Eulerian quasisymmetric functions $ Q_{n,j}$ and refine the chromatic symmetric functions of Stanley \cite{st4}.   One outcome of this study is the result that the joint distribution of $\maj$ and $\exc$ on $\sg_n$ is equidistributed with the joint distribution of a permutation statistic of Rawlings \cite{ra1} and $\des$.

The Schur positivity and unimodality conjectures of Section 5.2 (Conjectures 5.11 and 5.14) were recently proved in joint work with Anthony Henderson.  These results will appear in a forthcoming paper.  

\section*{Acknowledgements}
The research presented here began while both authors were visiting the Mittag-Leffler Institute as participants in a combinatorics program organized by Anders Bj\"orner and Richard Stanley.  We thank the organizers for inviting us to participate and the staff of the
Institute for their  hospitality and support.  We are also grateful to Ira Gessel for some very useful
discussions.


\begin{thebibliography}{xxxx}

\bibitem{ai} R. Askey and M. Ismail, {\it Permutation problems and special functions}, Canad. J. Math. {\bf 28} (1976), 853--894.

\bibitem{bs}  E. Babson and E. Steingr\'{\i}msson, {\it Generalized permutation patterns
and a classification of the Mahonian statistics}, S\'em.  Lothar.
Combin., {\bf B44b} (2000), 18 pp.

\bibitem{br} D. Beck and J.B. Remmel, {\it Permutation enumeration of the symmetric group and the combinatorics of symmetric functions},  J. Combin. Theory Ser. A  {\bf 72}  (1995), 1--49.


\bibitem{bhw} L.J. Billera, S.K. Hsiao and S. van Willigenburg, {\it Peak quasisymmetric functions and Eulerian enumeration}, Advances in Math.  {\bf 176}  (2003),   248--276.


\bibitem{bw} A. Bj\"orner and V. Welker,
{\it Segre and Rees products of posets, with ring-theoretic applications},
J. Pure Appl. Algebra {\bf 198} (2005),  43--55.


\bibitem{bra} P. Br\"and\'en, {\it Actions on permutations and unimodality of descent polynomials}, European J. Combinatorics {\bf 29} (2008), 514--531.

\bibitem{bre1} F. Brenti, {\it 
Unimodal polynomials arising from symmetric functions},
Proc. Amer. Math. Soc. {\bf 108} (1990),  1133--1141.

\bibitem{bre2} F. Brenti, {\it
Permutation enumeration symmetric functions, and unimodality},
Pacific J. Math. {\bf 157} (1993), 1--28. 

\bibitem{Br} K. S. Brown, {\it Buildings}, Springer-Verlag, New
York, 1989.


\bibitem{c} L. Carlitz, {\it A combinatorial property of $q$-Eulerian numbers}, The American Mathematical Monthly, {\bf 82} (1975), 51--54.


\bibitem{csv} L. Carlitz, R. Scoville, and T. Vaughan, {\it Enumeration of pairs of sequences by rises, falls and levels}, Manuscripta Math. {\bf 19} (1976),  211--243.

\bibitem{csz} R.J. Clarke, E.  Steingr\'{\i}msson, and J. Zeng,  {\it New
Euler-Mahonian statistics on permutations and words}, Adv. in Appl.
Math.  {\bf 18}  (1997),   237--270.

\bibitem{com} L. Comtet, Advanced Combinatorics, Reidel, Dordrecht, 1974.

\bibitem{dw2} J. D\'esarm\'enien and M.L. Wachs, {\it Descents des d\'erangements et mots circulaires},  S\'eminaire Lothairingen de Combinatoire Actes 19e, Publ. IRMA, Strasbourg, 1988.


\bibitem{dw} J. D\'esarm\'enien and M.L. Wachs,
{\it Descent classes of permutations with a given number of fixed points},
J. Combin. Theory Ser. A {\bf 64} (1993),  311--328.

\bibitem{dmp}
P. Diaconis, M. McGrath, J. Pitman, {\em
Riffle shuffles, cycles, and descents},
Combinatorica {\bf 15} (1995),  11--29.

\bibitem{dl} I. Dolgachev and V. Lunts,
{\it A character formula for the representation of a Weyl group in the
cohomology of the associated toric variety}, J. Algebra {\bf 168}
(1994),  741--772.

\bibitem{dgg} J. Dollhopf, I. Goulden, and  C. Greene, {\it Words avoiding a reflexive acyclic relation},
Electon. J. Combin. {\bf 11} (2006) \#R28.

\bibitem{eh} R. Ehrenborg, {\it On posets and Hopf algebras},  Adv. in Math. {\bf 119} (1996), 1-25.

\bibitem{foa3} D. Foata, {\it Distributions eul\'eriennes et mahoniennes
sur le groupe des permutations}, 
Higher combinatorics (Proc. NATO Advanced
Study Inst., Berlin, 1976), pp. 27--49, Reidel, Dordrecht-Boston,
Mass., 1977.



\bibitem{FH} D. Foata and G.-N. Han, {\it Fix Mahonian calculus III; A quadruple distribution}, Monatshefte f\"ur Mathematik, {\bf 154} (2008),  177--197.

\bibitem{FH2} D. Foata and G.-N. Han, {\it Signed words and permutations, V; a sextuple distribution}, Ramanujan J. {\bf 19} (2009), 29--52.

\bibitem{FH3}  D. Foata and G.-N. Han, {\it The q-tangent and q-secant numbers via basic Eulerian polynomials}, Proceedings of the American Mathematical Society, {\bf 138} (2010), 385--393.



\bibitem{fs} D. Foata  and M.-P. Sch\"utzenberger, {\it Th\'eorie g\'eom\'etrique des polynomes eul\'eriens}, Lecture Notes in Mathematics, Vol. 138 Springer-Verlag, Berlin-New York 1970.

\bibitem{fs2} D. Foata  and M.-P. Sch\"utzenberger,
{\it Major index and inversion number of permutations},
Math. Nachr. {\bf 83} (1978), 143--159.

\bibitem{fz} D. Foata and D. Zeilberger, {\it Denert's permutation
statistic is indeed Euler-Mahonian},  Stud. Appl. Math.  {\bf 83}
(1990), 31--59.

\bibitem{fu} W. Fulton, {\it Introduction to toric varieties}, Annals
of Mathematics Studies, 131, The William H. Roever Lectures in
Geometry, Princeton University Press, Princeton, NJ, 1993.

\bibitem{gar} A.M. Garsia, {\it On the maj and inv q-analogues of Eulerian polynomials}, J. Linear Multilinear Alg. {\bf 8} (1980), 21--34.


\bibitem{gg} A.M. Garsia and I. Gessel,
{\it Permutation statistics and partitions},
Advances in Math. {\bf 31} (1979),  288--305.

\bibitem{grem} A.M. Garsia and J.B. Remmel, {\it Q-counting rook configuration and a formula of Frobenius}, J. Combin. Theory Ser. A {\bf 41} (1986), 246--275.


\bibitem{gr} I.M. Gessel and C. Reutenauer,
{\it Counting permutations with given cycle structure and descent set},
J. Combin. Theory Ser. A {\bf 64} (1993), 189--215.

\bibitem{hag} J. Haglund, {\it q-Rook polynomials and matrices over finite
fields}, Adv. in Appl. Math. {\bf 20} (1998), 450-487.

\bibitem{h} P. Hanlon, {\it The action of $S_n$ on the components of the Hodge decomposition of Hochschild homology},   Michigan Math. J.  {\bf 37} (1990), 105--124.

\bibitem{hy1} M. Hyatt, {\it Eulerian quasisymmetric functions for the type B Coxeter group and other wreath product groups}, preprint, arXiv:1007.0459.

\bibitem{hy} M. Hyatt,  University of Miami Ph.D. dissertation, in preparation.

\bibitem{jo} J. Jonsson, {\it The Rees product of a Boolean algebra and a chain}, preprint.

\bibitem{kz} D. Kim and J. Zeng, {\it A new decomposition of derangements}, J. Combin. Theory Ser. A {\bf 96} (2001), 192--198.

\bibitem{kn} D. Knuth, The Art of Computer Programming, Vol. 3 Sorting and Searching, Second Edition, Reading, Massachusetts: Addison-Wesley, 1998.





\bibitem{mac1} P.A. MacMahon, Combinatory Analysis, 2 volumes,
Cambridge University Press, London, 1915-1916.  Reprinted by
Chelsea, New York, 1960.

\bibitem{mac2} P.A. MacMahon, {\it The indices of permutations and the
derivation therefrom of functions of a single variable associated
with the permutations of any assemblage of objects}, Amer. J. Math.
{\bf 35}  (1913),  no. 3, 281--322.



\bibitem{p} D. Perrin, {\it Factorizations of Free Monoids},  in M. Lothaire, Combinatorics on Words, Ch. 5, Encyclopedia of Math. and its Appl., Vol. 17, Addison-Wesley, Reading, MA, 1983.

\bibitem{pr} C. Procesi,
{\it The toric variety associated to Weyl chambers},
Mots, 153--161,
Lang. Raison. Calc.,
Herms, Paris, 1990.

\bibitem{rrw} A. Ram, J. Remmel, and T. Whitehead, {\it Combinatorics of the $q$-basis of symmetric functions},  J. Combin. Theory Ser. A  {\bf 76}  (1996),   231--271.

\bibitem{ra1} D. Rawlings, {\it The r-major index}, J. Combin. Theory Ser. A {\bf 23} (1981), 176--179.

\bibitem{ra} D. Rawlings,
{\it Enumeration of permutations by descents, idescents, imajor index, and basic components},
J. Combin. Theory Ser. A {\bf 36} (1984), 1--14.

\bibitem{re} V. Reiner,
{\it Signed permutation statistics and cycle type}, European J. Combin. {\bf 14} (1993),  569--579.

\bibitem{rsw} V. Reiner, D. Stanton and D. White, {\it The cyclic sieving phenomenon},
 J. Combin. Theory Ser. A {\bf 108} (2004),  17-50. 

\bibitem{rod} O. Rodrigues, {\it Note sur les inversions, ou derangements
produits dans les permutations}, Journal de Mathematiques {\bf 4} (1839),
236--240.




\bibitem{ssw} B. Sagan, J. Shareshian and M.L. Wachs, {\it Eulerian quasisymmetric functions and cyclic sieving}, Advances in Applied Math., to appear.

\bibitem{sw} J. Shareshian and M.L. Wachs, {\it  q-Eulerian polynomials: excedance number and major index}, Electron. Res. Announc. Amer. Math. Soc. {\bf 13} (2007), 33--45.

\bibitem{sw2} J. Shareshian and M.L. Wachs, {\it  Poset homology of Rees products and q-Eulerian polynomials}, Electron. J. Combin.  {\bf 16}  (2009), R20, 29 pp.

\bibitem{sw3} J. Shareshian and M.L. Wachs, {\it Rees products and lexicographical shellability}, in preparation.

\bibitem{sw4} J. Shareshian and M.L. Wachs, {\it Chromatic quasisymmetric functions}, in preparation.



\bibitem{sk} M. Skandera,
{\it An Eulerian partner for inversions},
S\'em. Lothar. Combin. {\bf 46} (2001/02), Art. B46d, 19 pp. (electronic).




\bibitem{st} R.P. Stanley,
{\it Ordered structures and partitions}, Memoirs Amer. Math. Soc.  {\bf 119} (1972).

\bibitem{st1} R.P. Stanley,
{\it Binomial posets, M\"obius inversion, and permutation enumeration},
J. Combinatorial Theory Ser. A {\bf 20} (1976), 336--356.

\bibitem{st2} R.P. Stanley,  {\it Log-concave and unimodal sequences in
algebra, combinatorics, and geometry},  Graph theory and its
applications: East and West (Jinan, 1986),  500--535, Ann. New
York Acad. Sci., 576, New York Acad. Sci., New York, 1989.

\bibitem{st4} R.P. Stanley, {\it A symmetric function generalization of the chromatic polynomial of a graph}, Advances in Math. {\bf 111} (1995), 166--194.

\bibitem{st5} R.P. Stanley,  Enumerative combinatorics, Vol. 1,
2nd ed.,
Cambridge Studies in Advanced Mathematics, {\bf 49}, Cambridge University Press, Cambridge, 1997.

\bibitem{st3} R.P. Stanley,  Enumerative combinatorics, Vol. 2,
Cambridge Studies in Advanced Mathematics, {\bf 62},
Cambridge University Press, Cambridge, 1999.


\bibitem{stem1} J.R. Stembridge, {\it Eulerian numbers, tableaux, and the
Betti numbers of a toric variety},  Discrete Math. {\bf 99} (1992),
307--320.

\bibitem{stem2} J.R. Stembridge, {\it Some permutation representations of
Weyl groups associated with the cohomology of toric varieties},
Adv. Math.  {\bf 106}  (1994),   244--301.


\bibitem{ace} {S. Veigneau}, {\it ACE, an Algebraic
                      Combinatorics Environment for the computer
                      algebra system MAPLE\/}, {\it User's Reference
                      Manual\/}, {\it Version 3.0\/}, IGM {\bf 98--11},
                      Universit\'e de Marne-la-Vall\'ee, 1998.

\bibitem{wa1} M.L. Wachs, {\it On $q$-derangement numbers}, Proc.  Amer. Math. Soc. {\bf 106} (1989), 273--278.


\bibitem{wa3} M.L. Wachs,
{\it The major index polynomial for conjugacy classes of permutations},
Discrete Math. {\bf 91} (1991), 283--293.

\bibitem{wa2} M.L. Wachs,
{\it An involution for signed Eulerian numbers},
Discrete Math. {\bf 99} (1992), 59--62.



\bibitem{w1} M.L. Wachs, {\it Poset topology: tools and applications}, Geometric Combinatorics,  IAS/PCMI
lecture notes series (E. Miller, V. Reiner, B. Sturmfels,  eds.), {\bf 13} (2007), 497--615.


\end{thebibliography}
\end{document}